\documentclass{lms}
\usepackage{amsmath,amssymb,amsrefs,amsfonts,eucal,yfonts,amscd}
\usepackage[dvipdf]{graphicx}

\newtheorem{theorem}{Theorem}[section]
\newtheorem{proposition}[theorem]{Proposition}
\newtheorem{lemma}[theorem]{Lemma}
\newtheorem{corollary}[theorem]{Corollary}
\newtheorem{proexample}[theorem]{Example}
\newnumbered{example}[theorem]{Example}
\newnumbered{definition}[theorem]{Definition}
\newnumbered{question}[theorem]{Question}
\newnumbered{remark}[theorem]{Remark}
\newnumbered{remarks}[theorem]{Remarks}
\newnumbered{observation}[theorem]{Observation}
\newnumbered{conjecture}[theorem]{Conjecture}

\renewcommand{\L}{{\mathcal L}}                                 
\newcommand{\I}{{\mathcal I}}                                   

\newcommand{\M}{{\mathcal M}}
\newcommand{\Real}{\mathbb R}                                   
\newcommand{\Natural}{{\mathbb N}}                              
\newcommand{\Rational}{{\mathbb Q}}                              
\newcommand{\1}{\boldsymbol{1}}

\def\ker{\mathop{\rm ker}\nolimits}
\def\card{\mathop{\rm card}\nolimits}
\newcommand{\carda}{\textswab{a}}
\newcommand{\cardb}{\textswab{b}}
\newcommand{\cardc}{\textswab{c}}
\begin{document}
\title{Free and Projective Banach Lattices.}
\author{B. de Pagter and A.W. Wickstead}
\classno{46B42}

 \maketitle
\begin{abstract}
We define and prove the existence of  free Banach lattices in the category of Banach lattices and contractive lattice homomorphisms and establish some of their fundamental properties. We give much more detailed results about their structure in the case that there are only a finite number of generators and give several Banach lattice characterizations of the number of generators being, respectively, one, finite or countable. We define a Banach lattice $P$ to be \emph{projective} if whenever $X$ is a Banach lattice, $J$ a closed ideal in $X$, $Q:X\to X/J$ the quotient map, $T:P\to X/J$ a linear lattice homomorphism and $\epsilon>0$ there is a linear lattice homomorphism $\hat{T}:P\to X$ such that (i) $T=Q\circ \hat{T}$ and (ii) $\|\hat{T}\|\le (1+\epsilon)\|T\|$.  We establish the connection between projective Banach lattices and free Banach lattices and describe several families of Banach lattices that are projective as well as proving that some are not.
\end{abstract}

\section{Introduction.}

Free and projective objects have not played anywhere near as
important a r\^{o}le in analysis as in algebra, nevertheless there
has been some work done on these objects, mainly with the results
that one would expect. For example, the existence of free and
projective Banach spaces is virtually folklore but is uninteresting
as both are of the form $\ell_1(I)$ for an arbitrary index set $I$.
The
existence of free vector lattices over an arbitrary number of
generators is also long established and holds no real surprises, see
\cite{Ba} or \cite{Bl} for details. In this note we investigate free
and projective Banach lattices. Some of our results are rather surprising
and although we are able to answer many questions we are forced to leave several unanswered.

It is almost obvious that, if it exists, then the free Banach
lattice over $\carda$ generators must be the completion of the free
vector lattice over $\carda$ generators for some lattice norm. That
the required norm actually exists is easily proved, but describing
it in concrete and readily identifiable terms is not so easy.
Indeed, except in the case $\carda=1$, it is not a classical Banach
lattice norm at all. In fact it is only in the case that $\carda$ is finite that the free Banach lattice
over $\carda$ generators is even isomorphic to an AM-space.

\S2 is primarily devoted to establishing notation whilst \S3
recapitulates the existing theory of free vector lattices. We then
prove the existence of free Banach lattices in \S4 and give a
representation on a compact Hausdorff space in \S5. We establish
some of the basic properties of free Banach lattices in \S6. The finitely generated free
 Banach lattices are by far the easiest ones to understand, and we
 investigate their structure in \S7. In \S8
we give some characterizations of free Banach lattices over, respectively, one, a finite
number or a countable number of generators,
 amongst all free Banach lattices. Preparatory to looking at
 projective Banach lattices, in \S9 we investigate when disjoint families
 in quotient Banach lattices $X/J$ can be lifted to disjoint families
 in $X$, giving a  positive result for countable families and a negative
 result for larger ones. We prove the connection between free and
 projective Banach lattices in \S10 and in \S11 find some classes of
 Banach lattices that are, or are not, projective. Finally, \S12
 contains some open problems.

 Let us emphasize at this point that this paper is set in the category of Banach lattices and linear lattice homomorphisms. 
There is a substantial theory of \emph{injective} Banach lattices (and indeed we refer to them later) but this is set in the context of Banach lattices 
and positive (or regular) operators.\footnote{In fact, although we can find no explicit proof in the literature, there is no non-zero injective in the category
of Banach lattices and linear lattice homomorphisms. Indeed, suppose that $F$ were a non-zero injective. Let $\carda$ be strictly greater than the cardinality 
of $F^*$ and let $\mu$
be the product of $\carda$ many copies  of the measure which assigns mass $\frac12$ to each of 0 and 1 in $\{0,1\}$. This is a homogenous measure space and each order interval
 in $\L_1(\mu)$ has the property that the least cardinality of a dense subset is precisely $\carda$, see \cite{Se} \S26 for details.
 In particular every order interval has cardinality at least $\carda$. As $\mu$ is finite, 
the same is true of $\L_\infty(\mu)$. 
Pick any non-zero $y\in F_+$. As $F$ is alleged to be injective, there is a linear lattice homomorphism $T$
 extending the map that 
takes the constantly one function in  $\L_1(\mu)$, $\1$,  to $y$. The adjoint of this  maps $F^*$ into 
$\L_1(\mu)^*=\L_\infty(\mu)$ and is interval preserving, \cite{MN}, 
Theorem 1.4.19. In particular, if $f\in F_+^*$ with $f(y)>0$  then $T^*f(\1)=f(T\1)=f(y)>0$, so  
the image of the order interval $[0,f]$ 
will be a non-zero order interval in $\L_\infty(\mu)$ which has cardinality at least $\carda$.
This contradicts the fact that $[0,f]$ has cardinality strictly less than $\carda$.}
Thus there is no reason to expect any kind of duality between the two notions.

\section{Notation.}\label{notation}

In this short section we establish the notation that we will use
concerning functions and function spaces. If $A$ and $X$ are
non-empty sets then, as usual, $X^A$ denotes the set of all maps
from $A$ into $X$. If $\emptyset\ne B\subseteq A$ then we let
$r_B:X^A\to X^B$ denote the restriction map with $r_B\xi=\xi_{|B}$
for $\xi\in X^A$. Clearly, $r_B$ is surjective. On occasions we will
also write $\xi_B$ in place of $r_B(\xi)$.

The space of all real-valued functions on $X^A$, $\Real^{X^A}$, is a
vector lattice under the pointwise operations. Again, we consider
the setting where $B$ is a non-empty subset of $A$ and define
$j_B:\Real^{X^B}\to \Real^{X^A}$  by $(j_Bf)(\xi)=f(\xi_B)$ for
$\xi\in X^A$ and $f\in \Real^{X^B}$. This makes $j_B$ an injective
lattice homomorphism. The following description of the image of
$j_B$ is easily verified.

\begin{lemma}\label{trivlemma}If $A,B$ and $X$ are non-empty sets with $B\subseteq A$
and $f\in \Real^{X^A}$ then the following are equivalent:
\begin{enumerate}
\item $f\in j_B(\Real^{X^B})$.
\item If $\xi, \eta\in X^A$ with $\xi_B=\eta_B$ then
$f(\xi)=f(\eta)$.
\end{enumerate}
\end{lemma}

We now specialize somewhat by assuming that $X\subseteq \Real$ and
that $0\in X$. This means that if $\xi\in X^A$, $\emptyset\ne B\subseteq A$ and $\chi_B$ is the characteristic function of $B$ then the pointwise
product $\xi\chi_B\in X^A$.

\begin{lemma}\label{restprops}If $\emptyset\ne B\subseteq A$ and $0\in
X\subseteq \Real$ then the map $P_B:\Real^{X^A}\to
j_B\big(\Real^{X^A}\big)$ defined by
\[(P_Bf)(\xi)=f(\xi\chi_B)\qquad(\xi\in X^A,\ f\in \Real^{X^A})\]
is a linear lattice homomorphism and a projection onto
$j_B\big(\Real^{X^B}\big)$. Furthermore, if $B_1, B_2\subseteq A$
are non-empty sets with non-empty intersection then
$P_{B_1}P_{B_2}=P_{B_2}P_{B_1}=P_{B_1\cap B_2}$.
\begin{proof} It is clear that $P_B$ is a well-defined vector
lattice homomorphism of $X^A$ into itself. If $\xi,\eta\in X^A$ are
such that $\xi_B=\eta_B$ then
$(P_Bf)(\xi)=f(\xi\chi_B)=f(\eta\chi_B)=(P_Bf)(\eta)$ so by Lemma
\ref{trivlemma} $P_Bf\in j_B(\Real^{X^B})$ for all $f\in \Real^{X^A}$. If $f\in
\Real^{X^B}$  then for any $\xi\in X^A$ we
have $P_B(j_Bf)(\xi)=(j_Bf)(\xi \chi_B)=(j_Bf)(\xi)$ as $\xi$ and
$\xi\chi_B$ coincide on $B$ and using Lemma \ref{trivlemma} again.
Thus $P_B$ is indeed a projection.

Finally, if $f\in \Real^{X^A}$ and $\xi\in X^A$ then
\begin{align*} P_{B_1}P_{B_2}f(\xi)&=(P_{B_2})(f\chi_{B_1})=f(\xi
\chi_{B_1}\chi_{B_2})\\
&=f(\xi \chi_{B_1\cap B_2})= (P_{B_1\cap B_2}f)(\xi),
\end{align*}
which shows that $P_{B_1} P_{B_2}=P_{B_1\cap B_2}$. Similarly
$P_{B_2} P_{B_2}=P_{B_2\cap B_1}=P_{B_1\cap B_2}$ and the proof is
complete.
\end{proof}
\end{lemma}

In future, we will identify $\Real^{X^B}$ with the vector sublattice
$j_B(\Real^{X^B})$ of $\Real^{X^A}$.

If $L$ is any vector lattice and $D$ a non-empty subset of $L$ then
$\langle D\rangle$ will denote the vector sublattice of $L$
generated by $D$. All elements of $\langle D\rangle$ can be obtained
from those of $D$ by the application of a finite number of
multiplications, additions, suprema and infima.
The following  simple consequence of this observation may also be proved directly:

\begin{lemma}If $L$ and $M$ are vector lattices, $T:L\to M$ is a vector lattice homomorphism and $\emptyset\ne
D\subseteq L$ then $\langle T(D)\rangle =T(\langle D\rangle)$.
\end{lemma}

We specialize further now to the case that $X=\Real$. On the space
$\Real^A$ we can consider the product topology, which is the
topology of pointwise convergence on $A$. By definition, this is the
weakest topology such that all the functions $\delta_a:\xi\mapsto
\xi(a)$ are continuous on $\Real^A$ for each $a\in A$. As a
consequence we certainly have $\langle \{\delta_a:a\in
A\}\rangle\subset C(\Real^A)$. In fact we can do rather better than
this. A function $f:\Real^A\to \Real$ is \emph{homogeneous} if
$f(t\xi)=tf(\xi)$ for $\xi\in \Real^A$ and $t\in[0,\infty)$. The
space $H(\Real^A)$ of continuous homogeneous real-valued functions
on $\Real^A$ is a vector sublattice of $C(\Real^A)$ and clearly
$\langle \{\delta_a:a\in A\}\rangle\subset H(\Real^A)$.

\section{Free Vector Lattices.}

In this section we recapitulate much of the theory of free vector
lattices, both to make this work as self-contained as possible and
in order to establish both our notation (which may not coincide with
that used in other papers on free vector lattices) and to point out
some properties that we will use later.

\begin{definition}If $A$ is a non-empty set then a \emph{free vector
lattice} over $A$ is a pair $(F,\iota)$ where $F$ is a vector
lattice and $\iota:A\to F$ is a map with the property that for any
vector lattice $E$ and any map $\phi:A\to E$ there is a unique
vector lattice homomorphism $T:F\to E$ such that $\phi=T\circ\iota$.
\end{definition}

It follows immediately from this definition that the map $\iota$
must be injective, as we can certainly choose $E$ and $\phi$ to make
$\phi$ injective. Many of the results that follow are almost
obvious, but we prefer to make them explicit.

\begin{proposition}\label{freegen}
If $(F,\iota)$ is a free vector lattice over $A$ then $F$ is
generated, as a vector lattice, by $\iota(A)$.
\begin{proof}
Let $G$ be the vector sublattice of $F$ generated by $\iota(A)$.
Define $\phi:A\to G$ by $\phi(a)=\iota(a)$ then it follows from the
definition that there is a unique vector lattice homomorphism
$T:F\to G$ with $T\big(\iota(a)\big)=\phi(a)=\iota(a)$ for $a\in A$.
If $j:G\to F$ is the inclusion map, then $j\circ T:F\to F$ is a
vector lattice homomorphism with $(j\circ
T)\big(\iota(a)\big)=j\big(\iota(a)\big)=\iota(a)$ for $a\in A$. The
identity on $F$, $I_F$, is also a vector lattice homomorphism from
$F$ into itself with $I_F(\iota(a))=\iota(a)$. The uniqueness part
of the definition of a free vector lattice applied to the map
$a\mapsto \iota(a)$, of $A$ into $F$, tells us that these two maps
are equal so that $j\circ T=I_F$ from which we see that $F\subseteq
G$ and therefore $F=G$ as claimed.
\end{proof}
\end{proposition}

The definition of a free vector lattice  make the following result easy to prove.

\begin{proposition}If $(F,\iota)$ and $(G,\kappa)$ are free vector
lattices over a non-empty  set $A$ then there is a (unique) vector
lattice isomorphism $T:F\to G$ such that
$T\big(\iota(a)\big)=\kappa(a)$ for $a\in A$.
\end{proposition}

In view of this we will just refer to a free vector lattice
$(F,\iota)$ over a set $A$ as \emph{the free vector lattice over
$A$} (or sometimes as the \emph{free vector lattice generated by
$A$} when we identify $A$ with a subset of that free vector
lattice). We will denote it by $FVL(A)$. It will be clear that if
$A$ and $B$ are sets of equal cardinality then $FVL(A)$ and $FVL(B)$
are isomorphic vector lattices, so that $FVL(A)$ depends only on the
cardinality of the set $A$. Thus we will also use the notation
$FVL(\carda)$ for $FVL(A)$ when $\carda$ is the cardinality of $A$.
This is the notation that will be found elsewhere in the literature.
We retain both versions so that we can handle proper inclusions of
$FVL(B)$ into $FVL(A)$ when $B\subset A$ even when $A$ and $B$ have
the same cardinality.

If $\iota:A\to FVL(A)$ is the embedding of $A$ into $FVL(A)$
specified in the definition then we will often write $\delta_a$ for
$\iota(a)$ and refer to the set $\{\delta_a:a\in A\}$ as the
\emph{free generators} of $FVL(A)$.

A slight rewording of the definition of a free vector lattice is
sometimes useful, which trades off uniqueness of the lattice
homomorphism for specifying that $\iota(A)$ is a generating set. The
proof of this follows immediately from results above.

\begin{proposition}\label{freealt}If $A$ is a non-empty set then the vector lattice
$F$ is the free vector lattice over $A$ if and only if
\begin{enumerate}
\item There is a subset $\{\delta_a:a\in A\}\subset F$, with
$\delta_a\ne \delta_b$ if $a\ne b$, which generates $F$ as a vector
lattice.
\item For every vector lattice $E$ and any family $\{x_a:a\in
A\}\subset E$ there is a vector lattice homomorphism $T:F\to E$ such
that $T(\delta_a)=x_a$ for $a\in A$.
\end{enumerate}
\end{proposition}

We will find the next simple result useful later.

\begin{proposition}\label{subfree} Let $A$ be a non-empty set and $\{\delta_a:a\in A\}$
be the free generators of $FVL(A)$. Let $B$ and $C$ be non-empty
subsets of $A$ with $B\cap C\ne \emptyset$. \begin{enumerate}
\item The vector sublattice of $FVL(A)$ generated by
$\{\delta_b:b\in B\}$ is (isomorphic to) the free vector lattice
$FVL(B)$.
\item There is a lattice homomorphism projection $P_B$ from $FVL(A)$
onto $FVL(B)$.
\item $P_C P_B=P_B P_C=P_{B\cap C}$.
\end{enumerate}
\begin{proof}
(i) Let $F$ denote the vector sublattice of $FVL(A)$ generated by
$\{\delta_b:b\in B\}$. Suppose that $E$ is a vector lattice and
$\pi:B\to E$ is any map. There is a unique vector lattice
homomorphism $T:FVL(A)\to E$ with $T(\delta_b)=\pi(b)$ for $b\in
B$ and $T(\delta_a)=0$ for $a\in A\setminus B$. The restriction $S$
of $T$ to $F$ gives us a vector lattice homomorphism $S:F\to E$ with
$S(\delta_b)=\pi(b)$. It follows from Proposition \ref{freealt}
that $F=FVL(B)$.

(ii) The free property of $FVL(A)$ gives a (unique) lattice
homomorphism $P_B:FVL(A)\to FVL(A)$ with $P_B(\delta_b)=\delta_b$ if
$b\in B$ and $P_B(\delta_a)=0$ if $a\in A\setminus B$. As $P_B$ maps
the generators of $FVL(A)$ into $FVL(B)$, we certainly have
$P_B\big(FVL(A)\big)\subseteq FVL(B)$. Also, $P_B$ is the identity
on the generators of $FVL(B)$ so is the identity linear operator on
$FVL(B)$ so that $P_B$ is indeed a projection.

(iii) If $a\in B\cap C$ then $P_CP_B\delta_a=P_BP_C\delta_a=P_{B\cap
C}\delta_a=\delta_a$ whilst if $a\notin B\cap C$ then
$P_CP_B\delta_a=P_BP_C\delta_a=P_{B\cap C}\delta_a=0$. Thus the
three vector lattice homomorphisms $P_B P_C$, $P_C P_B$ and
$P_{B\cap C}$ coincide on a set of  generators of $FVL(A)$ and are
therefore equal.
\end{proof}
\end{proposition}

So far all our discussions of free vector lattices have been rather
academic as we have not shown that they exist. However it was shown
in \cite{Ba} (see also \cite{Bl}) that they do exist. In essence we
have:

\begin{theorem}\label{FVLexists} For any non-empty set $A$, $FVL(A)$ exists and is
the vector sublattice of $\Real^{\Real^A}$ generated by $\delta_a$
($a\in A$) where $\delta_a(\xi)=\xi(a)$ for $\xi\in \Real^A$.
\end{theorem}

It is reasonable to ask how this representation of $FVL(A)$
interacts with the properties of free vector lattices noted above.
With the notation of \S\ref{notation}, if $\emptyset\ne B\subseteq
A$ then the map $j_B:\Real^{\Real^B}\to \Real^{\Real^A}$  is a
vector lattice embedding of $\Real^{\Real^B}$ into
$\Real^{\Real^A}$. This corresponds precisely to the embedding of
$FVL(B)$ into $FVL(A)$ as indicated in Proposition \ref{subfree}. If
we use $\delta_a$ to denote the map $\xi\mapsto \xi(a)$ on $\Real^A$
and $\eta_b$ for the map $\xi\mapsto \xi(b)$ on $\Real^B$ then we
have, for $b\in B$ and $\xi\in\Real^A$
\[(j_B\eta_b)(\xi)=\eta_b(\xi_B)=\xi(b)=\delta_b(\xi)\]
so that $j_B\eta_b=\delta_b$. We know from \S\ref{notation} that
$j_B$ is a vector lattice homomorphism so that $j_B\big(FVL(B)\big)$
is the vector sublattice of $FVL(A)$ generated by $\{\delta_b:b\in
B\}$ which is precisely what was described in Proposition
\ref{subfree}.

Also, if $B\subseteq A$ then we may consider $FVL(B)\subseteq
FVL(A)\subseteq \Real^{\Real^A}$. The projection map $P_B:FVL(A)\to
FVL(B)$ defined in Proposition \ref{subfree} (2) is then precisely
the restriction to $FVL(A)$ of the projection
$P_B:\Real^{\Real^A}\to \Real^{\Real^B}$ described in Lemma \ref{restprops}. We
will temporarily denote this projection by $\tilde{P}_B$ to
distinguish it from the abstract projection. Once we establish
equality that distinction will not be required and we will omit the
tilde. As $P_B$ and $\tilde{P}_B$ are both vector lattice
homomorphisms it suffices to prove this equality for the generators
of $FVL(A)$. If $b\in B$ then
\[(\tilde{P}_B\delta_b)(\xi)=\delta_b(\xi\chi_B)=(\xi\chi_B)(b)=\xi(b)=\delta_b(\xi)\]
for $\xi\in \Real^A$ so that
$\tilde{P}_B\delta_b=\delta_b=P_B\delta_b$. If, on the other hand,
$a\in A\setminus B$ then
\[(\tilde{P}_B\delta_a)(\xi)=\delta_a(\xi\chi_B)=0\]
for $\xi\in \Real^A$ so that $\tilde{P}_B\delta_a=0=P_B\delta_a$.

A few more observations will be of use later.

\begin{proposition}\label{freefinite}If $A$ is a non-empty set and $\mathcal{F}(A)$ denotes the collection of all non-empty
finite subsets of $A$, then
\[FVL(A)=\bigcup_{B\in \mathcal{F}(A)} FVL(B).\]
\begin{proof}Any element of $FVL(A)$ is in the vector sublattice of
$FVL(A)$ generated by a finite number of generators
$\{\delta_{a_1},\delta_{a_2},\dots,\delta_{a_n}\}$ so lies in
$FVL(\{a_1,a_2,\dots,a_n\})$.
\end{proof}
\end{proposition}

\begin{proposition}\label{finiteou}If $A$ is a finite set then $\sum_{a\in A}|\delta_a|$ is a strong order unit for $FVL(A)$.
\begin{proof}Obvious as $FVL(A)$ is generated by the set
$\{\delta_a:a\in A\}$.
\end{proof}
\end{proposition}

\begin{lemma}\label{pointeval}The real valued vector lattice homomorphisms on
$FVL(A)$  are precisely the evaluations at points of $\Real^A$.
\begin{proof}
It is clear that if $\xi\in\Real^A$ then the map
$\omega_\xi:f\mapsto f(\xi)$ is a real valued vector lattice
homomorphism on $\Real^{\Real^A}$ and therefore on $FVL(A)$. Note,
in particular, that $\omega_\xi(\delta_a)=\delta_a(\xi)=\xi(a)$.
 Conversely, if $\omega$ is a real valued vector
lattice homomorphism on $FVL(A)$ then we may define $\xi\in \Real^A$
by $\xi(a)=\omega(\delta_a)$ for $a\in A$. Now we see that for this
$\xi$, $\omega_\xi$ is a real valued vector lattice homomorphism on
$FVL(A)$ with $\omega_\xi(\delta_a)=\xi(a)=\omega(\delta_a)$. The
two maps $\omega$ and $\omega_\xi$ coincide on a set of  generators
of $FVL(A)$ so, being vector lattice homomorphisms, are equal.
\end{proof}
\end{lemma}

\section{Free Banach Lattices.}

\begin{definition}If $A$ is a non-empty set then a \emph{free Banach
lattice} over $A$ is a pair $(X,\iota)$ where $X$ is a Banach
lattice and $\iota:A\to X$ is a bounded map with the property that for any
Banach lattice $Y$ and any bounded map $\kappa:A\to Y$ there is a unique
vector lattice homomorphism $T:X\to Y$ such that $\kappa=T\circ\iota$ and $\|T\|=\sup\{\|\kappa(a)\|:a\in A\}$.
\end{definition}

It is clear that the set $\{\iota(a):a\in A\}$ generates $X$ as a Banach lattice (cf Proposition \ref{freegen}).

\begin{proposition}The definition forces each $\iota(a)$ to have norm precisely one. For if $\kappa(a)=1\in \Real$ for each
$a\in A$ then the map $T$ that is guaranteed to exist has norm 1, so
that $1=\|T\big(\iota(a)\big)\|\le \|\iota(a)\|$. On the other hand,
if we take $\kappa=\iota$, then $T$ is identity operator, with norm
1, so that $\sup\{\|\iota(a)\|:a\in A\}=1$.
\end{proposition}

\begin{proposition}If $(X,\iota)$ and $(Y,\kappa)$ are free Banach lattices over a non-empty set $A$ then there is a (unique) isometric order isomorphism
$T:X\to Y$ such that $T\big(\iota(a)\big)=\kappa(a)$ for $a\in A$.
\begin{proof} As $(X,\iota)$ is free, there is a vector lattice homomorphism $T:X\to Y$ with $T\big(\iota(a)\big)=\kappa(a)$ for
$a\in A$ with $\|T\|=\sup\{\|\kappa(a)\|:a\in A\}=1$, by the
preceding proposition. There is similarly a contractive vector
lattice homomorphism $S:Y\to X$ with
$S\big(\kappa(a)\big)=\iota(a)$. By uniqueness, the compositions $S\circ T$ and
$T\circ S$ must be
the identity operators. This suffices to prove our claim.
\end{proof}
\end{proposition}

Similarly to the free vector lattice case, we use the notation $FBL(A)$ for the free Banach lattice over $A$ if it exists (which we will shortly show is the case.)
Since we know that if $A$ and $B$ have the same cardinality then $FBL(A)$ and $FBL(B)$ are isometrically order isomorphic, we will also
use the notation $FBL(\carda)$ to denote a free Banach lattice on a set of cardinality $\carda$.
Again, we will also use the notation $\delta_a$ for $\iota(a)$ and refer to $\{\delta_a:a\in A\}$ as the \emph{free generators} of $FBL(A)$.

Our first task is to show
that free Banach lattices do indeed exist.

\begin{definition}
If $A$ is a non-empty set then we will define a mapping from $FVL(A)^\sim$ into the extended non-negative reals by
\[\|\phi\|^\dag=\sup\{|\phi|(|\delta_a|):a\in A\}.\]
We define also
\[FVL(A)^\dag=\{\phi\in FVL(A)^\sim:\|\phi\|^\dag<\infty\}\]
which it is clear  is a vector lattice ideal in the Dedekind complete vector lattice $FVL(A)^\sim$.
\end{definition}

Suppose that a positive functional $\phi$ vanishes on each $|\delta_a|$. Each element $x$ of $FVL(A)$ lies in the sublattice of $FVL(A)$ generated by a finite set of generators $\{a_k:1\le k\le n\}$. By Proposition \ref{finiteou} $e=\sum_{k=1}^n |\delta_{a_k}|$ is a strong order unit for that sublattice. Thus there is $\lambda\in\Real$ with $|x|\le \lambda e$ so that $|\phi(x)|\le \phi(|x|)\le \phi(\lambda e)=\lambda \sum_{k=1}^n \phi(|\delta_{a_k}|)=0$ and thus $\phi=0$. It is now
clear that $\|\cdot\|^\dag$ is a lattice norm on $FVL(A)^\dag$.
Given the embedding of $FVL(A)$ in $\Real^{\Real^A}$ given in
Theorem \ref{FVLexists}, if $\xi\in \Real^A$ then $\omega_\xi\in FVL(A)^\dag$ if and only if the map $\xi:A\to \Real$ is bounded and then 
$\|\omega_\xi\|^\dag=\sup_{a\in A} |\xi(a)|$. By Lemma \ref{pointeval}, these maps are
lattice homomorphisms. Note that if $A$ is an infinite set then there is an unbounded $\xi\in \Real^A$ which induces 
$\omega_\xi\in FVL(A)^\sim\setminus FVL(A)^\dag$.

\begin{definition}
For $f\in FVL(A)$, where $A$ is a non-empty set, define
\[\|f\|_F=\sup\{\phi(|f|):\phi\in FVL(A)^\dag_+, \|\phi\|^\dag\le 1\}.\]
\end{definition}

\begin{proposition}For any non-empty set $A$, $\|\cdot\|_F$ is a lattice norm on $FVL(A)$.
\begin{proof}Our first step is to show that $\|\cdot\|_F$ is real-valued. By Proposition \ref{freefinite}, any $f\in FVL(A)$
actually lies in $FVL(B)$ for some finite subset $B\subseteq A$. By
Proposition \ref{finiteou}, $FVL(B)$ has a strong order unit
$\sum_{b\in B}|\delta_b|$, so there is $\lambda$ with $|f|\le \lambda
\sum_{b\in B}|\delta_b|$. If $\phi\in FVL(A)^\dag_+$ with
$\|\phi\|^\dag\le 1$ then
\[\phi(|f|)\le \phi\left(\lambda \sum_{b\in B}|\delta_b|\right)=\lambda \sum_{b\in B}\phi(|\delta_b|)\le \lambda \sum_{b\in B}1\]
so that $\|f\|_F$ is certainly finite.

If $\|f\|_F=$ then $\phi(|f|)=0$ for all $\phi\in FVL(A)^\dag_+$. Using the observation above, $f(\xi)=\omega_\xi(f)=0$ for any bounded function $\xi:A\to\Real$. But there is a finite set $B\subset A$ such that $f\in FVL(B)$, so that $f(\xi)=f(\xi\chi_B)$ for all $\xi\in \Real^A$. As each $\xi\chi_B$ is bounded, $f(\xi)=0$ for all $\xi\in \Real^A$ and therefore $f=0$.

That $\|\cdot\|_F$ is sublinear and positively homogeneous are
obvious, so that $\|\cdot\|_F$ is a norm on $FVL(A)$, which is
clearly a lattice norm.
\end{proof}
\end{proposition}

Note in particular that we certainly have  $\|\delta_a\|_F=1$ for all $a\in A$. In fact, this construction gives us our desired free Banach
lattices.

\begin{theorem}\label{freeexists}For any non-empty set $A$, the pair consisting of the completion of
$FVL(A)$, under the norm $\|\cdot\|_F$, and the map $\iota:a\to
\delta_a$, is the free Banach lattice over $A$.
\begin{proof}
Suppose that $Y$ is any Banach lattice and $\kappa :A\to Y_1$, the
unit ball of $Y$. There is a vector lattice homomorphism
$T:FVL(A)\to Y$ with $T\big(\iota(a)\big)=\kappa(a)$ for all $a\in
A$, as $FVL(A)$ is free. We claim that if $f\in FVL(A)$ with
$\|f\|_F\le 1$ then $\|Tf\|=\big\||Tf|\|=\big\|T(|f|)\big\|\le 1$ in
$Y$, where we have the used that fact that the norm in $Y$ is a
lattice norm and that $T$ is a lattice homomorphism. If this were
not the case then we could find $\psi\in Y^*_{1+}$, a positive
linear functional on $Y$ with norm at most 1, with
$\psi\big(T(|f|)\big)>1$. As
$\|T\big(\iota(a)\big)\|=\|\kappa(a)\|\le 1$ for all $a\in A$, we
have $\big\||T(\iota(a)|\big\|=\big\|T(|\iota(a)|)\big\|\le 1$,
using again the fact that $T$ is a lattice homomorphism. Thus
$\Big|\psi\Big(T\big(|\iota(a)|\big)\Big)\Big|\le 1$ for all $a\in
A$. Using the functional $\psi\circ T$ in the definition of
$\|f\|_F$, we see that $\|f\|_F\ge \psi\big(T(|f|)\big)>1$,
contradicting our assumption that $\|f\|_F\le 1$.

The completion of $FVL(A)$ is a Banach lattice and $T$ will extend
by continuity to it whilst still taking values in $Y$ as $Y$ is
complete.
\end{proof}
\end{theorem}


We will eventually need to know the relationship between different
free Banach lattices, so we record now the following result.

\begin{proposition}\label{sublattice}
If $B$ is a non-empty subset of $A$ then $FBL(B)$ is isometrically
order isomorphic to the closed sublattice of $FBL(A)$ generated by
$\{\delta_b:b\in B\}$. Furthermore there is a contractive lattice
homomorphic projection $P_B$ of $FBL(A)$ onto $FBL(B)$.
\begin{proof}
Recall from Proposition
\ref{subfree} that $FVL(B)$ is isomorphic to the sublattice of
$FVL(A)$ generated by $\{\delta_b:b\in B\}$ that there is a lattice
homomorphism projection $P_B$ of $FVL(A)$ onto $FVL(B)$ with
$P_B(\delta_b)=\delta_b$ if $b\in B$ and $P_B(\delta_a)=0$ if $a\in
A\setminus B$. As $\|\delta_b\|_F=1$ in both $FBL(A)$ and $FBL(B)$, there are contractive lattice homomorphisms of $FBL(B)$ into $FBL(A)$ and of $FBL(B)$ onto $FBL(A)$ which act into the same way on the generators so extend these. The conclusion is now clear.
\end{proof}
\end{proposition}

There is also a simple relationship between their duals. 
This is a consequence of the following result which is surely well known but for which we can find no convenient reference, 
but see \cite{TL}, IV.12, Problem 6 and \cite{DS}, Lemma VI.3.3 for similar results.

\begin{proposition}
If $P$ is a contractive lattice homomorphism projection from a Banach lattice $X$ onto a closed sublattice $Y$ then $P^*X^*$ is a weak$^*$-closed band in $X^*$ which is isometrically order isomorphic to $Y^*$.
\begin{proof}
Write $\ker(P)$ for the kernel of $P$, which is a lattice ideal in $X$, and $Z=\{\phi\in X^*:\phi_{|\ker(P)}\equiv 0\}$, which is a weak$^*$-closed band in $X^*$. It is clear that $P^*X^*=Z$.

Define $J:Y^*\to X^*$ by $J\phi=\phi\circ P$ and note that $J:Y^*\to Z$ with $\|J\|\le \|P\|$. If $\phi\in Z$ then $J(\phi_{|Y})=\phi$ so that $J$ is actually an isometry of $Y^*$ onto $Z$. It is clear that both $J$ and $J^{-1}$ are positive. Thus $J:Y^*\to P^*X^*$ is actually an isometric order isomorphism.
\end{proof}
\end{proposition}

\begin{corollary}\label{dualembedding}If $B$ is a non-empty subset of $A$ then $FBL(B)^*$ is isometrically order isomorphic to a weak$^*$-closed band in $FBL(A)^*$.
\end{corollary}

As in the algebraic case, if $B$ and $C$ are two  subsets of $A$
with $B\cap C\ne \emptyset$ of $A$ then $P_B P_C=P_C P_B=P_{B\cap
C}$.

In particular, the embedding of the finitely generated free closed
sublattices are important.

\begin{proposition}\label{finitestrong}Let $\mathcal{F}(A)$ be the collection of all non-empty 
finite subsets of $A$, ordered by inclusion. The net of projections
$\{P_B:B\in\mathcal{F}(A)\}$ in $FBL(A)$ converges strongly to the
identity in $FBL(A)$.
\begin{proof}
If $f\in FVL(A)$ then there is actually $B_0\in \mathcal{F}(A)$ with
$P_B(f)=f$ whenever $B_0\subset B$. Recall that each $P_B$ is a
contraction. If $\epsilon>0$ and $f\in FBL(A)$, choose $f'\in
FVL(A)$ with $\|f-f'\|_F<\epsilon/2$ and then $B_0\in
\mathcal{F}(A)$ with $P_B(f')=f'$ for $B_0\subset B$. Then if
$B_0\subset B$ then
\[\|P_Bf-f\|_F\le
\|P_Bf-P_Bf'\|+\|P_Bf'-f'\|_F+\|f'-f\|_F<\epsilon,\] which completes
the proof.
\end{proof}
\end{proposition}

Before looking at some properties of $FBL(A)$ in detail, we will ask
about its normed dual.

\begin{proposition}\label{whatisdual}If $A$ is any non-empty set then the three normed spaces $(FVL(A)^\dag, \|\cdot\|^\dag)$, $(FVL(A),||\cdot\|_F)^*$ and $FBL(A)^*$ are isometrically order isomorphic.
\begin{proof} If $\phi\in FBL(A)^*$ then the restriction map $\phi\mapsto \phi_{|FVL(A)}$ is an order isomorphism, by continuity,
and as $\|\delta_a\|=1$ we have $|\phi|(\delta_a)\|\le \|\phi\|$ so that $\|\phi_{|FVL(A)}\|^\dag\le \|\phi\|$.
On the other hand, as each $\|\delta_a\|=1$ we see that
\begin{align*}
\|\phi\|=\big\||\phi|\big\|&=\sup\{|\phi|(f):\|f\|_F\le 1\}\\
&\le \sup\{|\phi|(|\delta_a|):a\in A\}=\|\phi_{|FVL(A)|}\|^\dag
\end{align*}
so the isometric order isomorphism of the first and third spaces is proved. The identification of the second and third follows from the density of $FVL(A)$ in $FBL(A)$.
\end{proof}
\end{proposition}

As we noted above, if $A$ is infinite then $FVL(A)^\dag\ne FVL(A)^\sim$. On the other hand, we have:

\begin{proposition}If $n\in\Natural$ then $FBL(n)^*$ is isometrically order isomorphic to the whole of $FVL(n)^\sim$ under
the norm $\|\cdot\|^\dag$.
\begin{proof}All that remains to establish is that $\|\phi\|^\dag$ is finite for all $\phi\in FVL(n)$. Given that
$\|\phi\|^\dag$ is, in this case, a finite supremum of real values $|\phi|(|\delta_a|)$, this is clear.
\end{proof}
\end{proposition}

\section{A smaller representation space}

The set $\Delta _|=[ -1,1]^A$ is a compact subset of $\Real^A$. We call a function $f:\Delta _A\rightarrow \Real$
\emph{homogeneous} if $f( t\xi) =tf( \xi)$ for $\xi \in \Delta _A$ and $t\in [ 0,1] $ (this is consistent with
the definition for functions on $\Real^A$). The space of continuous
homogeneous real-valued functions on $\Delta _A$ is denoted by $H(\Delta _A) $. If we equip $C( \Delta _A) $ with the
supremum norm $\left\Vert \cdot \right\Vert _{\infty }$, then $H\left(
\Delta _{A}\right) $ is a closed vector sublattice of $C\left( \Delta
_{A}\right) $ (and hence $H\left( \Delta _{A}\right) $ is itself a Banach
lattice with respect to this norm).

\begin{lemma}
\label{BLem01}The restriction map $R:H\left( \mathbb{R}^{A}\right)
\rightarrow H\left( \Delta _{A}\right) $ is a injective vector lattice
homomorphism.
\end{lemma}

\begin{proof}
The only part of the proof that is not completely trivial is that the map $R$
is injective. Suppose that $f\in H\left( \mathbb{R}^{A}\right) $ and $Rf=0$.
If $\xi \in \mathbb{R}^{A}$, consider the net $\left\{ \xi \chi _{B}:B\in
\mathcal{F}\left( A\right) \right\} $, where $\mathcal{F}\left( A\right) $
is the collection of all non-empty finite subsets of $A$ ordered by
inclusion, then we have $\xi \chi _{B}\rightarrow _{\mathcal{F}\left(
A\right) }\xi $ in $\mathbb{R}^{A}$. For any $B\in \mathcal{F}\left(
A\right) $, there is $t>0$ such that $t\xi \chi _{B}\in \left[ -1,1\right]
^{A}$, so that $tf\left( \xi \chi _{B}\right) =f\left( t\xi \chi _{B}\right)
=0$ by homogeneity. Hence, $f\left( \xi \chi _{B}\right) =0$ and so $f\left(
\xi \right) =0$ by the continuity of $f$, so that $f=0$.
\end{proof}

It should be noted that the restriction map is not surjective unless $A$ is
a finite set.

\begin{example}
It suffices to prove the non-surjectiveness in the case that $A=\mathbb{N}$.
Define $g\in H\left( \Delta _{\mathbb{N}}\right) $ by $g\left( \xi \right)
=\sum_{k=1}^{\infty }2^{-k}\xi \left( k\right) $ for $\xi \in \Delta _{%
\mathbb{N}}$. Suppose that there is $f\in H\left( \mathbb{R}^{\mathbb{N}%
}\right) $ with $Rf=g$. Define $\eta \in \mathbb{R}^{\mathbb{N}}$ by $\eta
\left( k\right) =2^{k}$ and let $\eta _{n}=\eta \chi _{\left\{ 1,\ldots
,n\right\} }$, for $n\in \mathbb{N}$, so that $\eta _{n}\rightarrow \eta $
in $\mathbb{R}^{\mathbb{N}}$. But, for each $n\in \mathbb{N}$ we have
\[
f\left( \eta _{n}\right) =2^{n}f\left( 2^{-n}\eta _{n}\right) =2^{n}g\left(
2^{-n}\eta _{n}\right) =n.
\]%
As $f$ is supposed to be continuous, this is impossible.
\end{example}

Note that this example also shows that the space $H\left( \mathbb{R}%
^{A}\right) $, equipped with the sup-norm over $\Delta _{A}$, is not
complete if $A$ is infinite. This is one of the reasons that we shall use
the space $H\left( \Delta _{A}\right) $.

In general, $FVL\left( A\right) $ may be identified with a vector sublattice
of $H\left( \mathbb{R}^{A}\right) $ (see Theorem \ref{FVLexists}), which in turn,
courtesy of Lemma \ref{BLem01}, may be identified with a vector sublattice
of $H\left( \Delta _{A}\right) $ via the restriction map $R$. This
identification extends to $FBL\left( A\right) $. The proof of this turns out
to be slightly more tricky than might have been anticipated.

For sake of convenience, we denote by $J=J_{A}$ the restriction to $%
FVL\left( A\right) $ of the restriction map $R:H\left( \mathbb{R}^{A}\right)
\rightarrow H\left( \Delta _{A}\right) $. Since $\left\Vert J\delta
_{a}\right\Vert _{\infty }=1$ for all $a\in A$, it is clear that $\left\Vert
J\right\Vert =1$ and so, $\left\Vert Jf\right\Vert _{\infty }\leq \left\Vert
f\right\Vert _{F}$ for all $f\in FVL\left( A\right) $. Since $H\left( \Delta
_{A}\right) $ is a Banach lattice with respect to $\left\Vert \cdot
\right\Vert _{\infty }$, $J$ extends by continuity to a lattice homomorphism
$J:FBL\left( A\right) \rightarrow H\left( \Delta _{A}\right) $ with $%
\left\Vert J\right\Vert =1$. Note that, by the universal property of $%
FBL\left( A\right) $, $J$ is the unique lattice homomorphism from $FBL\left(
A\right) $ into $H\left( \Delta _{A}\right) $ satisfying $J\delta
_{a}=\delta _{a|\Delta _{A}}$, $a\in A$. This implies, in particular,
that if $B$ is a non-empty subset of $A$, then $J_{B}$ is the restriction of
$J_{A}$ to $FBL\left( B\right) $ (cf. Proposition \ref{sublattice}). The problem is to
show that this extension $J$ is injective.

First, we consider the situation that $A$ is finite, in which case
everything is very nice indeed.

\begin{proposition}
\label{BProp01}For any non-empty finite set $A$, the map $J:FBL\left(
A\right) \rightarrow H\left( \Delta _{A}\right) $ is a surjective norm and
lattice isomorphism.
\end{proposition}

\begin{proof}
We claim that $\left\Vert f\right\Vert _{F}\leq n\left\Vert Jf\right\Vert
_{\infty }$, $f\in FVL\left( A\right) $, where $n$ is the cardinality of $A$. Indeed, if $f\in FVL\left( A\right)
$, then
\[
\left\vert f\right\vert \leq \left\Vert Jf\right\Vert _{\infty
}\bigvee\nolimits_{a\in A}\left\vert \delta _{a}\right\vert
\]%
and so,
\[
\left\Vert f\right\Vert _{F}\leq \left\Vert Jf\right\Vert _{\infty
}\left\Vert \bigvee\nolimits_{a\in A}\left\vert \delta _{a}\right\vert
\right\Vert _{F}\leq \left\Vert Jf\right\Vert _{\infty }\sum\nolimits_{a\in
A}\left\Vert \delta _{a}\right\Vert _{F}=n\left\Vert Jf\right\Vert _{\infty
}.
\]%
This proves the claim. Consequently, $\left\Vert Jf\right\Vert _{\infty
}\leq \left\Vert f\right\Vert _{F}\leq n\left\Vert Jf\right\Vert _{\infty }$%
, $f\in FVL\left( A\right) $, which implies that $J:FBL\left( A\right)
\rightarrow H\left( \Delta _{A}\right) $ is a norm and lattice isomorphism.
It remains to be shown that $J$ is surjective. For this purpose, denote by $%
S_{A}$ the compact subset of $\Delta _{A}$ given by $S_{A}=\left\{ \xi \in
\Delta _{A}:\left\Vert \xi \right\Vert _{A}=1\right\} $. Since $A$ is
finite, the restriction map $r:H\left( \Delta _{A}\right) \rightarrow
C\left( S_{A}\right) $ is a surjective norm and lattice isomorphism. Since
the functions $\left\{ \delta _{a|S_{A}}:a\in A\right\} $ separate
the points of $S_{A}$, it follows via the Stone-Weierstrass theorem that $%
\left( r\circ J\right) \left( FBL\left( A\right) \right) =C\left(
S_{A}\right) $ and hence $J\left( FBL\left( A\right) \right) =H\left(
\Delta _{A}\right) $. The proof is complete.
\end{proof}

This norm isomorphism is not an isometry unless $n=1$. In fact, if $%
a_{1},\ldots ,a_{n}\in A$ are distinct, then $\left\Vert
\bigvee\nolimits_{j=1}^{n}\left\vert \delta _{a_{j}}\right\vert \right\Vert
_{F}=n$ (indeed, consider the lattice homomorphism $T:FBL\left( A\right)
\rightarrow \ell _{1}^{n}$ satisfying $T\left( \delta _{a_{j}}\right) =e_{j}$%
, $1\leq j\leq n$, where $e_{j}$ denotes the $j$-th unit vector in $\ell _{1}^{n}$).

Sometimes it will be convenient to use the following, slightly weaker,
description.

\begin{corollary}
For any non-empty finite set $A$, $FBL\left( A\right) $ is linearly order
isomorphic to $H\left( \mathbb{R}^{A}\right) $.
\end{corollary}

\begin{proof}
We only need to observe that the restriction map $R:H\left( \mathbb{R}%
^{A}\right) \rightarrow H\left( \Delta _{A}\right) $ is onto whenever $A$ is
finite.
\end{proof}

To show that the lattice homomorphism $J:FBL\left( A\right) \rightarrow
H\left( \Delta _{A}\right) $ is injective in general, we will make use of
real-valued linear lattice homomorphisms on $FBL\left( A\right) $ in the
course of proving this and it will later allow us to characterize these in
general, which must be worth knowing anyway!

\begin{theorem}
\label{BThm01}If $A$ is a non-empty set, then $\omega :FBL\left( A\right)
\rightarrow \mathbb{R}$ is a lattice homomorphism if and only if there
exists $\xi \in \Delta _{A}$ and $0\leq \lambda \in \mathbb{R}$ such that $%
\omega \left( f\right) =\lambda Jf\left( \xi \right) $ for all $f\in
FBL\left( A\right) $.
\end{theorem}

\begin{proof}
If $\omega $ is a real valued lattice homomorphism on $FBL\left( A\right) $,
then it follows from Lemma \ref{pointeval} that there is $\eta \in \mathbb{R}^{A}$
such that $\omega \left( f\right) =f\left( \eta \right) $, $f\in FVL\left(
A\right) $. As $FBL\left( A\right) $ is a Banach lattice, $\omega $ is $%
\left\Vert \cdot \right\Vert _{F}$-bounded and so, $\sup_{a\in A}\left\vert
\eta \left( a\right) \right\vert =\sup_{a\in A}\left\vert \omega \left(
\delta _{a}\right) \right\vert =\left\Vert \omega \right\Vert <\infty $.
Hence, there is a $\lambda >0$ such that $\xi =\lambda ^{-1}\eta \in \Delta
_{A}$. If $f\in FVL\left( A\right) $, then
\[
\omega \left( f\right) =f\left( \eta \right) =\lambda f\left( \lambda
^{-1}\eta \right) =\lambda Jf\left( \xi \right) .
\]

Given $f\in FBL\left( A\right) $, choose a sequence $\left( g_{n}\right) $
in $FVL\left( A\right) $ with\break $\left\Vert f-g_{n}\right\Vert _{F}\rightarrow
0$, so that $\left\Vert Jf-Jg_{n}\right\Vert _{\infty }\rightarrow 0$ and
hence $Jg_{n}\left( \xi \right) \rightarrow Jf\left( \xi \right) $. Thus,
\[
\omega \left( f\right) =\lim_{n\rightarrow \infty }\omega \left(
g_{n}\right) =\lambda \lim_{n\rightarrow \infty }Jg_{n}\left( \xi \right)
=\lambda Jf\left( \xi \right) .
\]

The converse is clear as if $\xi \in \Delta _{A}$ and $0\leq \lambda \in
\mathbb{R}$, then the formula $\omega \left( f\right) =\lambda Jf\left( \xi
\right) $, $f\in FBL\left( A\right) $, defines a lattice homomorphism on $%
FBL\left( A\right) $.
\end{proof}

It is clear already that, for $f\in FVL\left( A\right) $, $f=0$ if and only
if $Jf=0$. if and only if $\omega \left( f\right) =0$ for every $\left\Vert
\cdot \right\Vert _{F}$-bounded real-valued lattice homomorphism on $%
FVL\left( A\right) $. We need this equivalence for $f\in FBL\left( A\right) $%
.

\begin{corollary}\label{freefinal}
For any non-empty set $A$ and $f\in FBL\left( A\right) $ the following are
equivalent:

\begin{enumerate}
\item[(i)] $f=0$;

\item[(ii)] $\omega \left( f\right) =0$ for all real-valued lattice
homomorphisms on $FBL\left( A\right) $;

\item[(iii)] $Jf=0$.
\end{enumerate}
\end{corollary}

\begin{proof}
Clearly, (i) implies (iii) and that (iii) implies (ii) follows directly from
Theorem \ref{BThm01}.

Now assume that (ii) holds. Note firstly that it follows from Proposition 
\ref{BProp01} that for any non-empty finite subset $B\subseteq A$ the
restriction of $J$ to $FBL\left( B\right) $ is injective. For such a set $B$%
, the map $g\longmapsto \left( JP_{B}g\right) \left( \xi \right) $, $g\in
FBL\left( A\right) $, is a real-valued lattice homomorphism on $FBL\left(
A\right) $ for each $\xi \in \Delta _{A}$, so that $JP_{B}g=0$. As $J$ is
injective on $FBL\left( B\right) $, this shows that $P_{B}f=0$. It follows
from Proposition \ref{finitestrong} that $P_{B}f\rightarrow f$ for $\left\Vert \cdot \right\Vert _{F}$, so that $f=0 $. This suffices to complete the proof.
\end{proof}

\begin{corollary}\label{fnrepn}
If $A$ is any non-empty set, then the lattice homomorphism $J:FBL\left(
A\right) \rightarrow H\left( \Delta _{A}\right) $ is injective, so that $%
FBL\left( A\right) $ is linearly order isomorphic to a vector sublattice of $%
H\left( \Delta _{A}\right) $.
\end{corollary}

In the sequel, we shall identify $FBL\left( A\right) $ with the vector
sublattice \linebreak $J\left( FBL\left( A\right) \right) $ of $H\left(
\Delta _{A}\right) $.

As we have seen in [Proposition 4.8], if $B$ is a non-empty subset of $A$,
then $FBL\left( B\right) $ may be identified isometrically with the closed
vector sublattice of $FBL\left( A\right) $ generated by $\left\{ \delta
_{b}:b\in B\right\} $ and there is a canonical contractive lattice
homomorphic projection $P_{B}$ in $FBL\left( A\right) $ onto $FBL\left(
B\right) $. It should be observed that we have the following commutative
diagram:
\[\begin{CD}
FBL(A)@>J_A>> H(\Delta_A)\\
@A k_B AA @AA j_B A\\
FBL(B) @>>J_B> H(\Delta_B)\\
\end{CD}
\]
where $j_{B}$ is the restriction to $H\left( \Delta _{B}\right) $ of the
injective lattice homomorphism $j_{B}$ introduced in Section 2, and $k_{B}$
is the isometric lattice embedding of $FBL\left( B\right) $ into $FBL\left(
A\right) $ guaranteed by Proposition \ref{sublattice}. Note that also $j_{B}$ is an
isometry. The commutativity of the diagram follows by considering the action
of the maps on the free generators of $FBL\left( B\right) $. Consequently,
the canonical embedding of $FBL\left( B\right) $ into $FBL\left( A\right) $
is compatible with the canonical embedding of $H\left( \Delta _{B}\right) $
into $H\left( \Delta _{A}\right) $. The next proposition describes this in
terms of $FBL\left( A\right) $ considered as a vector sublattice of $H\left(
\Delta _{A}\right) $. We consider $\mathbb{R}^{\Delta _{B}}$ as a subspace
of $\mathbb{R}^{\Delta _{A}}$ as explained in Section 2.

Recall that if $B$ is a non-empty subset of $A$, then for any $\xi \in
\Delta _{A}$ we denote by $\xi _{B}$ the restriction of $\xi $ to $B$, so
that $\xi _{B}\in \Delta _{B}$.

\begin{proposition}
\label{BProp02}Suppose that $B$ is a non-empty subset of $A$. Considering $%
FBL\left( A\right) $ as a vector sublattice of $H\left( \Delta _{A}\right) $%
, we have:

\begin{enumerate}
\item[(i)] the canonical projection $P_{B}$ of $FBL\left( A\right) $ onto $%
FBL\left( B\right) $ is given by $P_{B}f\left( \xi \right) =f\left( \xi \chi
_{B}\right) $, $\xi \in \Delta _{A}$, for all $f\in FBL\left( A\right) $;

\item[(ii)] if $f\in FBL\left( A\right) $, then a necessary and sufficient
condition for $f$ to belong to $FBL\left( B\right) $ is that $f\left( \xi
\right) =f\left( \eta \right) $ whenever $\xi ,\eta \in \Delta _{A}$ with $%
\xi _{B}=\eta _{B}$.
\end{enumerate}
\end{proposition}

\begin{proof}
(i). Let $P_{B}$ be the canonical projection in $FBL\left( A\right) $ onto $%
FBL\left( B\right) $ (see Proposition \ref{sublattice}), so that $P_{B}\delta
_{a}=\delta _{a}$ if $a\in B$ and $P_{B}\delta _{a}=0$ if $a\in A\diagdown B$%
. If $f\in FVL\left( A\right) $, then it follows from the observations
preceding Proposition \ref{freefinite} that $P_{B}f\left( \xi \right) =f\left( \xi \chi
_{B}\right) $, $\xi \in \Delta _{A}$. Given $f\in FBL\left( A\right) $, let $%
\left( f_{n}\right) $ be a sequence in $FVL\left( A\right) $ such that $%
\left\Vert f-f_{n}\right\Vert _{F}\rightarrow 0$, which implies that $%
\left\Vert f-f_{n}\right\Vert _{\infty }\rightarrow 0$ and so, $f_{n}\left(
\xi \right) \rightarrow f\left( \xi \right) $, $\xi \in \Delta _{A}$.
Furthermore, $\left\Vert P_{B}f-P_{B}f_{n}\right\Vert _{F}\rightarrow 0$ and
hence $P_{B}f_{n}\left( \xi \right) \rightarrow P_{B}f\left( \xi \right) $,
$\xi \in \Delta _{A}$. Since $P_{B}f_{n}\left( \xi \right) =f_{n}\left( \xi
\chi _{B}\right) \rightarrow f\left( \xi \chi _{B}\right) $, we may conclude
that $P_{B}f\left( \xi \right) =f\left( \xi \chi _{B}\right) $, $\xi \in
\Delta _{A}$.

(ii). \textit{Necessity}. If $f\in FBL\left( B\right) $ and $\xi ,\eta \in
\Delta _{A}$ are such that $\xi _{B}=\eta _{B}$, then $\xi \chi _{B}=\eta
\chi _{B}$ and hence it follows from (i) that
\[
f\left( \xi \right) =P_{B}f\left( \xi \right) =f\left( \xi \chi _{B}\right)
=f\left( \eta \chi _{B}\right) =P_{B}f\left( \eta \right) =f\left( \eta
\right) .
\]

\textit{Sufficiency}. If $f\in FBL\left( A\right) $ is such that $f\left(
\xi \right) =f\left( \eta \right) $ whenever $\xi ,\eta \in \Delta _{A}$
with $\xi _{B}=\eta _{B}$, then $P_{B}f\left( \xi \right) =f\left( \xi \chi
_{B}\right) =f\left( \xi \right) $, as $\left( \xi \chi _{B}\right) _{B}=\xi
_{B}$, for all $\xi \in \Delta _{A}$ and hence $f=P_{B}f\in FBL\left(
B\right) $. \medskip
\end{proof}

Recall that a sublattice $H$ of a lattice $L$ is said to be \textit{%
regularly embedded} if every subset of $H$ with a supremum (resp. infimum)
in $H$ has the same supremum (resp. infimum) in $L$. If we are dealing with
vector lattices it suffices to consider only the case of a subset of $H$
that is downward directed in $H$ to $0$ and check that it also has infimum $%
0 $ in $L$.

\begin{proposition}\label{orderclosedsublattice}
If $A$ is any non-empty set and $B$ is a non-empty subset of $A$, then $%
FBL\left( B\right) $ is regularly embedded in $FBL\left( A\right) $.
\end{proposition}

\begin{proof}
Suppose that $\left( f_{\gamma }\right) _{\gamma \in \Gamma }$ is a downward
directed net in $FBL\left( B\right) $ such that $f_{\gamma }\downarrow
_{\gamma }0$ in $FBL\left( B\right) $ and suppose that $g\in FBL\left(
A\right) $ satisfies $0<g\leq f_{\gamma }$ for all $\gamma \in \Gamma $. Let
$\xi _{0}\in \Delta _{A}$ be such that $g\left( \xi _{0}\right) >0$. We claim that we may assume that $\xi _{0}\chi _{B}\neq 0$. If our chosen $\xi_0$ is such that $\xi_0\chi_B=0$, i.e. $xi_0=\xi_0\chi_{A\setminus B}$, then consider $\xi_\epsilon=\xi_0+\epsilon \xi_B$. Since $xi_\epsilon\to \xi_0$ in $\Delta_A$ as $\epsilon\downarrow 0$ and $g$ is continuous we may choose $\epsilon\in(0,1]$ with $g(\xi_\epsilon)>0$ and then replace $xi_0$ by this $\xi_\epsilon$. 
 Given $b\in
B$, define $h\in H\left( \Delta _{A}\right) $ by setting
\[
h\left( \xi \right) =g\left( \xi \chi _{B}+\frac{\left\vert \xi \left(
b\right) \right\vert }{\left\Vert \xi _{0}\chi _{B}\right\Vert _{\infty }}%
\xi _{0}\chi _{A\diagdown B}\right) ,\ \ \ \xi \in \Delta _{A}.
\]%
We claim that $h\in FBL\left( A\right) $. Indeed, define the lattice
homomorphism $T:H\left( \Delta _{A}\right) \rightarrow H\left( \Delta
_{A}\right) $ by setting
\[
Tf\left( \xi \right) =f\left( \xi \chi _{B}+\frac{\left\vert \xi \left(
b\right) \right\vert }{\left\Vert \xi _{0}\chi _{B}\right\Vert _{\infty }}%
\xi _{0}\chi _{A\diagdown B}\right) ,\ \ \ \xi \in \Delta _{A},
\]%
for all $f\in H\left( \Delta _{A}\right) $. Observing that
\[
T\delta _{a}=\delta _{a}\chi _{B}\left( a\right) +\frac{\left\vert \delta
_{b}\right\vert }{\left\Vert \xi _{0}\chi _{B}\right\Vert _{\infty }}\delta
_{a}\left( \xi _{0}\right) \chi _{A\diagdown B}\left( a\right) ,
\]%
it follows that $T\delta _{a}\in FVL\left( A\right) $ for all $a\in A$ and
that $\sup_{a\in A}\left\Vert T\delta _{a}\right\Vert _{F}<\infty $.
Consequently, there exists a unique lattice homomorphism $S:FBL\left(
A\right) \rightarrow FBL\left( A\right) $ such that $S\delta _{a}=T\delta
_{a}$ for all $a\in A$. Evidently, $Tf=Sf$ for all $f\in FVL\left( A\right) $%
. Given $f\in FBL\left( A\right) $, we may approximate $f$ with a sequence $%
\left( f_{n}\right) $ with respect to $\left\Vert \cdot \right\Vert _{F}$.
Using that convergence with respect to $\left\Vert \cdot \right\Vert _{F}$
implies pointwise convergence on $\Delta _{A}$, it follows that $Sf=Tf$ (cf.
the proof of Proposition \ref{BProp02}). This implies, in particular, that $%
h=Tg=Sg\in FBL\left( A\right) $, by which our claim is proved.

If $\xi ,\eta \in \Delta _{A}$ are such that $\xi _{B}=\eta _{B}$, then $%
h\left( \xi \right) =h\left( \eta \right) $ and so, by Proposition \ref%
{BProp02} and Lemma \ref{trivlemma}, it follows that $h\in FBL\left( B\right) $. If $%
\xi \in \Delta _{A}$, then
\[
\xi _{B}=\left( \xi \chi _{B}+\frac{\left\vert \xi \left( b\right)
\right\vert }{\left\Vert \xi _{0}\chi _{B}\right\Vert _{\infty }}\xi
_{0}\chi _{A\diagdown B}\right) _{B}
\]%
(recall that the subscript $B$ indicates taking the restriction to the
subset $B$) and hence
\begin{eqnarray*}
f_{\gamma }\left( \xi \right)  &=&f_{\gamma }\left( \xi \chi _{B}+\frac{%
\left\vert \xi \left( b\right) \right\vert }{\left\Vert \xi _{0}\chi
_{B}\right\Vert _{\infty }}\xi _{0}\chi _{A\diagdown B}\right)  \\
&\geq &g\left( \xi \chi _{B}+\frac{\left\vert \xi \left( b\right)
\right\vert }{\left\Vert \xi _{0}\chi _{B}\right\Vert _{\infty }}\xi
_{0}\chi _{A\diagdown B}\right) =h\left( \xi \right) ,\ \ \ \xi \in \Delta
_{A},
\end{eqnarray*}%
that is, $f_{\gamma }\geq h\geq 0$ for all $\gamma \in \Gamma $. We may
conclude that $h=0$.

It follows, in particular, that
\[
g\left( \xi _{0}\chi _{B}+\frac{\left\vert \xi _{0}\left( b\right)
\right\vert }{\left\Vert \xi _{0}\chi _{B}\right\Vert _{\infty }}\xi
_{0}\chi _{A\diagdown B}\right) =0,\ \ \ b\in B.
\]%
Applying this to $b=b_{n}$, where $\left( b_{n}\right) $ is a sequence in $B$
satisfying $\left\vert \xi _{0}\left( b_{n}\right) \right\vert \rightarrow
\left\Vert \xi _{0}\chi _{B}\right\Vert _{\infty }$, the continuity of $g$
implies that
\[
g\left( \xi _{0}\right) =g\left( \xi _{0}\chi _{B}+\xi _{0}\chi _{A\diagdown
B}\right) =0,
\]%
which is a contradiction. The proof is complete.
\end{proof}

\section{Some Properties of Free Banach lattices.}

If $X$ is a non-empty set and $f:X\to\Real$ then we let $O_f=\{x\in
X:f(x)\ne 0\}$ and if $W$ is a non-empty subset of $\Real^X$ then we
define $O_W=\bigcup\{O_f:f\in W\}$. Although probably well known we
know of no convenient reference for the following result.

\begin{proposition}If $X$ is a Hausdorff topological space, $L$ a
vector sublattice of $C(X)$ and the open set $O_L$ is connected then
the only projection bands in $L$ are $\{0\}$ and $L$.
\begin{proof}
Suppose that $B$ is a projection band in $L$, so that $L=B\oplus B^d$. If $f\in B$ and $g\in B^d$ then $f\perp g$ and hence $O_f\cap O_g=\emptyset$ and therefore $O_B\cap O_{B^d}=\emptyset$.
Given $x\in O_L$ there is $0\ne f\in L_+$ with $f(x)>0$. We may write $f=f_1\oplus f_2$ with $0\le f_1\in B$ and $0\le f_2\in B^d$. Clearly, either $f_1(x)>0$ or
$f_2(x)>0$. I.e. $x\in O_{f_1}\cup O_{f_2}\subset O_B\cup O_{B^d}$. Hence $O_L\subset O_B\cup O_{B^d}$ and therefore $O_L= O_B\cup O_{B^d}$. The sets
$O_B$ and $O_{B^d}$ are both open and disjoint and $O_L$ is, by hypothesis, connected. This is only possible if either $O_B$ or $O_{B^d}$ is empty which says that either
$L=B^d$ or $L=B$.
\end{proof}
\end{proposition}

\begin{corollary}If $|A|\ge 2$ then the only projection bands in $FBL(A)$ are $\{0\}$ and $FBL(A)$.
\begin{proof}By Corollary \ref{fnrepn} we may identify $FBL(A)$ with a vector sublattice of $H(\Delta_A)\subset C(\Delta_A)$. Observe that
\[O_{FBL(A)}\supset \bigcup_{a\in A}O_{\delta_a}=\bigcup_{a\in A} \{\xi\in \Delta_A:\xi(a)\ne 0\}=\Delta_A\setminus\{0\}.\]
Clearly,  $O_{FBL(A)}\subset \Delta_A\setminus\{0\}$ so that $O_{FBL(A)}= \Delta_A\setminus\{0\}$ which, provided $|A|\ge 2$, is (pathwise) connected.
\end{proof}
\end{corollary}

\begin{corollary}If $|A|\ge 2$ then $FBL(A)$ is not Dedekind $\sigma$-complete.
\end{corollary}

\begin{corollary}If $|A|\ge 2$ then $FBL(A)$ has no atoms.
\begin{proof}
The linear span of an atom is always a projection band.
\end{proof}
\end{corollary}

\begin{corollary}If $a\in A$ then $\delta_a$ is a weak order unit for $FBL(A)$.
\begin{proof}If $f\in FBL(A)$ and $f\perp \delta_a$ then $O_f\subset \{\xi\in\Delta_A:\xi(a)=0\}$, and the latter set has an empty interior so that $O_f=\emptyset$
and hence $f=0$.
\end{proof}
\end{corollary}

\begin{corollary}\label{ctbledisjt}Every disjoint system in $FBL(A)$ is at most countable.
\begin{proof}If $\{u_i:i\in I\}$ is a disjoint family of strictly positive elements of $FBL(A)$ then the corresponding sets $O_{u_i}$ are non-empty disjoint subsets of $\Delta_A$.
As $\Delta_A=[-1,1]^A$ is a product of separable spaces, Theorem 2 of \cite{RS} tells us that $\Delta_A$ can contain only countably many disjoint non-empty open sets
so that the families
of all $O_{u_i}$ and of all $u_i$ are indeed countable.
\end{proof}
\end{corollary}

The same result is true for $FVL(A)$, being first proved by Weinberg in \cite{We2}. It can also be found, with essentially the current proof, in \cite{Ba}.

Recall that an Archimedean vector lattice is \emph{order separable} if every subset $D\subset L$ contains an at most countable subset with the same upper 
bounds in $L$ as $D$ has. This is equivalent to every order bounded disjoint family of non-zero elements being at most countable, \cite{LZ}, Theorem 29.3.
 Corollary \ref{ctbledisjt}
thus actually tells us that the universal completion of $FBL(A)$, \cite{LZ}, Definition 50.4,   is always order separable.

Every Banach lattice is a quotient of a free Banach lattice. We can actually make this statement quite precise. The following lemma is well known dating back, in the case that $\carda=\aleph_0$, to a result of Banach and Mazur \cite{BM}. A more accessible proof, again in the case that $\carda=\aleph_0$ (although the modifications needed for the general case are minor), are given as part of the proof of Theorem 5 of Chapter VII of \cite{Di}.

\begin{lemma}\label{BSlemma}Let $X$ be a Banach space and $D$ a dense subset of the unit ball of $X$. If $x\in X$ and $\|x\|<1$ then there are sequences $(x_n)$ in $D$ and $(\alpha_n)$ in $\Real$
such that $\sum_{n=1}^\infty |\alpha_n|<1$ and $x=\sum_{n=1}^\infty \alpha_n x_n$.
\end{lemma}

\begin{proposition}\label{freequotients}Let $X$ be a Banach lattice. If $D$ is a dense subset of the unit ball of $X$ of cardinality $\carda$, then there is 
a closed
ideal $J$ in $FBL(\carda)$ such that $X$ is isometrically order isomorphic to $FBL(\carda)/J$.
\begin{proof}Let $D=\{x_a:a\in\carda\}$. By the definition of a free Banach lattice there is a unique contractive lattice homomorphism
$T:FBL(\carda)\to X$ with $T(\delta_a)=x_a$ for each $a\in \carda$. If $x\in X$ with $\|x\|<1$ then Lemma \ref{BSlemma} gives us sequences $(x_{a_n})$ in $D$ and
$(\alpha_n)$ in $\Real$ with $\sum_{n=1}^\infty |\alpha_n|<1$ and $x=\sum_{n=1}^\infty \alpha_n x_{a_n}$. If we define $f\in FBL(\carda)$ by
$f=\sum_{n=1}^\infty\alpha_n \delta_{a_n}$, noting that this series converges absolutely, then $\|f\|_F<1$ and $Tf=x$. This shows that $T$ maps the open unit ball in $FBL(\carda)$
onto the open unit ball in $X$. In particular, $T$ is surjective.

Take $J$ to be the kernel of $T$ and let $Q:FBL(\carda)\to FBL(\carda)/J$ be the quotient map. Let $U:FBL(\carda)/J\to X$ be defined by $U(Qf)=Tf$ for
$f\in FBL(\carda)$, which is clearly well-defined. It is also clear that $U$ is a contractive lattice isomorphism. As $T$ maps the open unit ball of $FBL(\carda)$
onto the open unit ball of $X$ and $Q$ maps the open unit ball of $FBL(\carda)$ onto the open unit ball of $FBL(\carda)/J$, it follows that $U$ maps the open unit ball of $FBL(\carda)/J$
onto the open unit ball of $X$ so that $U$ is an isometry.
\end{proof}
\end{proposition}

\begin{corollary}Let $X$ be a Banach lattice. If $D$ is a dense subset of the unit ball of $X$ of cardinality $\carda$, then $FBL(\carda)^*$ contains a weak$^*$-closed band which is isometrically order isomorphic to $X^*$.
\begin{proof}
If $T:FBL(\carda)\to X$ is the quotient map from Proposition \ref{freequotients} then $T^*:X^*\to FBL(\carda)^*$ is an isometry and its range,
which is $\ker(T)^\perp$, is a weak$^*$-closed band. As $T$ is a surjective lattice homomorphism, $T^*$ is actually a lattice isomorphism.
\end{proof}
\end{corollary}

In particular note:

\begin{corollary}\label{ellinftyindual}
If $\carda$ is any  cardinal then there is a weak$^*$-closed band in $FBL(\carda)^*$ which is isometrically order isomorphic to $\ell_\infty(\carda)$.
\begin{proof}
If $\carda$ is infinite then we need merely note that the unit ball of  $\ell_1(\carda)$ has a dense subset of cardinality $\carda$ and that $\ell_\infty(\carda)$ may be identified with $\ell_1(\carda)^*$.

Suppose that $\card(A)=\carda$ is finite. For $a\in A$ we will write $\xi_a$ for
that element of $\Delta_A=[-1,1]^A$ with $\xi_a(a)=1$ and $\xi_a(b)=0$
if $a\ne b$. If $b\in A$ then
$|\delta_b|(\xi_a)=|\delta_b(\xi_a)|=1$ if $a=b$ and is zero if
$a\ne b$. It follows from the Proposition \ref{whatisdual} that the
functional $f\mapsto f(\xi_a)$ is a lattice homomorphism on $FBL(A)$, and therefore an atom of $FBL(A)^*$, of norm one. Finite sums of such maps also have norm one. This embeds a copy of $\ell_\infty(A)$ isometrically onto an order ideal in $FBL(A)^*$ which, as it is finite dimensional, is certainly a weak$^*$-closed band.
\end{proof}

\end{corollary}
\begin{corollary}\label{sepfreequotients}If $X$ is a separable Banach lattice then $X$ is isometrically order isomorphic to a Banach lattice quotient of $FBL(\aleph_0)$ and $X^*$ is isometrically order isomorphic to a weak$^*$-closed band in $FBL(\aleph_0)^*$.
\end{corollary}

This illustrates quite effectively what a rich structure free Banach lattices and their duals have. For example if $X$ and $Y$ are separable Banach lattices such that no two non-zero bands in $X^*$ and $Y^*$ are isometrically isomorphic then the isometrically order isomorphic bands in $FBL(\aleph_0)^*$ must be disjoint in the lattice theoretical sense. So, for example, we have:

\begin{corollary}In $FBL(\aleph_0)^*$ there are mutually disjoint weak$^*$-closed bands $A$ and $B_p$ ($p\in (1,\infty]$) with $B_p$ isometrically order isomorphic to $L_p([0,1])$ and $A$ to $\ell_\infty$.
\end{corollary}

This gives continuum many disjoint non-zero elements in $FBL(\aleph_0)^*$, which should be contrasted with Corollary \ref{ctbledisjt}.

\section{The Structure of Finitely Generated Free Banach Lattices.}

We will see shortly that $FBL(n)$ is not an
AM-space unless $n=1$, but it does have a lot of AM-structure provided that
$n$ is finite.

If we have only a finite number of generators, $n$, say then we may
identify $FBL(n)$ with $H(\Delta_n)$, where $\Delta_n$ is now a
product of $n$ copies of $[-1,1]$. In this setting, it might be more
useful to consider the restriction of these homogeneous functions to
the union of all the proper faces of $\Delta_n$, which we will
denote by $F_n$. An alternative description of this set is that it is
the points in $\Real^n$ with supremum norm equal to 1. Each of the
generators $\delta_k$ ($1\le k\le n$) takes the value +1 on one
maximal proper face of $F_n$ of dimension $n-1$ and the value $-1$
on the complementary face. These faces exhaust the maximal proper
faces of $\Delta_n$. The restriction map from $H(\Delta_n)$ to
$C(F_n)$ is a surjective vector lattice isomorphism and an isometry from the
supremum norm over $\Delta_n$ to the supremum norm over $F_n$. We
know also that these norms are equivalent to the free norm. Thus
when we identify $FBL(n)$ with $C(F_n)$, even though the norms are
not the same, the closed ideals, band, quotients etc remain the same
so that we can read many structural results off from those for
$C(K)$ spaces. Whenever we refer to the free norm on $C(F_n)$ we
refer to the free norm generated using the generators which take
value $\pm1$ on the maximal proper faces.

In particular, we may identify the dual of $FBL(n)$ with the space
of regular Borel measures on $F_n$, $\M(F_n)$. We will see in Theorem \ref{onegen} that unless $n=1$ the
the dual of the free norm, $\|\cdot\|^\dag$ is definitely not the
usual norm, $\|\cdot\|_1$, under which $\M(F_n)$ is an AL-space.
However, there remains a lot of AL-structure in this dual.

\begin{proposition}\label{L1idealdual}If $\mu\in\M(F_n)$ is supported by a maximal
proper face of $\Delta_n$ then $\|\mu\|^\dag=\|\mu\|_1$.
\begin{proof}Suppose first that $\mu\ge 0$.
Let the free generators be denoted by $\delta_1,
\delta_2,\dots,\delta_n$. If $G$ is the maximal proper face in
question, we may suppose that $G\subset|\delta_1|^{-1}(1)$. As
$|\delta_k|\le \1$ on $F_n$, for $1\le k\le n$, we have
\[\left|\int \delta_k\;d\mu\right|\le \int |\delta_k|\;d\mu\le \int
\1\;d\mu=\|\mu\|_1,\] and on taking the maximum we have
$\|\mu\|^\dag\le \|\mu\|_1$. On the other have, $|\delta_1|\equiv 1$
on $G$ so that
\[\|\mu\|^\dag\ge \int|\delta_1|\;d\mu=\|\mu\|_1,\]
so we have equality. Both $\|\cdot\|_1$ and $\|\cdot\|^\dag$ are
lattice norms, so in the general case we have
\[\|\mu\|^\dag=\big\||\mu|\big\|^\dag=\big\||\mu|\big\|_1=\|\mu\|_1\]
and the proof is complete.
\end{proof}
\end{proposition}

\begin{corollary}\label{AMsubset}If $f\in C(F_n)$ and there is a maximal proper face
$G$ such that $f$ vanishes off $G$ then $\|f\|_F=\|f\|_\infty$.
\begin{proof}
If $\mu\in\M(F_n)$ then we may write $\mu=\mu_G+\mu_{F_n\setminus
G}$, where $\mu_A(X)=\mu(A\cap X)$, and note that $\int
f\;d\mu=\int f\;d\mu_G$. If $\|\mu\|^\dag\le 1$ then
$\|\mu_G\|^\dag=\|\mu_G\|_1\le 1$ as $|\mu_G|\le |\mu|$. Thus
\begin{align*}\|f\|_F&=\sup\{\int |f|\;d|\mu|:\|\mu\|^\dag\le 1\}\\
&\le \sup\{\int |f|\;d|\mu|:\|\mu\|_1\le 1\}=\|f\|_\infty
\end{align*}
\end{proof}
\end{corollary}

This means that certain closed ideals in $FBL(n)$ are actually AM-spaces,
namely those that may be identified with functions on $F_n$ which
vanish on a closed set $A$ whose complement is contained in a single
proper face of $F_n$. Rather more interesting is an analogous result
for quotients.

In general, if $J$ is a closed ideal in a Banach lattice $X$ then
$(X/J)^*$ may be identified, both in terms of order and norm, with
the ideal $J^\circ=\{f\in X^*:f_{|J}\equiv 0\}$. We know that if $A$ is a closed subset of a compact
Hausdorff space $K$ and $J^A$ denotes the closed ideal $J^A=\{f\in
C(K):f_{|A}\equiv 0\}$ then when $C(K)$ is given the supremum norm
the normed quotient $C(K)/J^A$ is isometrically order isomorphic to
$C(A)$ under its supremum norm and its dual is isometrically order
isomorphic to the space of measures on $K$ which are supported by
$A$. In the particular case that $K=F_n$ we may still identify
quotients algebraically in the same way, but the description of the
quotient norm has to be modified slightly. That means that the
quotient norm may be described in a similar manner to our original
description of the free norm:

\begin{proposition}
If $A$ is a closed subset of $F_n$ and $C(F_n)$ is normed by its
canonical free norm then $C(F_n)/J^A$ is isometrically order
isomorphic to $C(A)$, where $C(A)$ is normed by
\[\|f\|_A=\sup\{\int |f|d|\mu_A|:\|\mu\|^\dag\le 1\}.\]
In this supremum we may restrict to measures $\mu$ supported by $A$.
\end{proposition}

In particular we have, using Proposition \ref{L1idealdual}:

\begin{corollary}\label{nicequotients}If $A$ is a closed subset of a proper face of $F_n$
and $C(F_n)$ is normed by its canonical free norm then $C(F_n)/J^A$
is isometrically order isomorphic to $C(A)$ under its supremum norm.
\end{corollary}

The free vector lattices over a finite number of generators exhibit
a lot of symmetry. For example it is not difficult to see that
$FVL(n)$ is invariant under rotations. In studying symmetry of
$FBL(n)$ it makes things clearer to identify $FBL(n)$ with the space
$C(S^{n-1})$ rather than $C(F_n)$, where $S^{n-1}$ is the Euclidean
unit sphere in $\Real^n$, even though the description of the free
norm is made slightly more difficult. In the case $n=2$, we are
looking at continuous functions on the unit circle and the dual free
norm is given by
\[\|\mu\|^\dag=\int_{S^1}|\sin(t)|\;d|\mu|(t)\vee
\int_{S^1}|\cos(t)|\;d|\mu|(t).\] In particular, if $\eta_x$
denotes the unit measure concentrated at $x$ then
\[\|\eta_x\|^\dag=|\sin(x)|\vee|\cos(x)|\]
which is certainly not rotation invariant. Note also that
\[\|\eta_x+\eta_{x+\pi/2}\|^\dag=\big(|\sin(x)|+|\sin(x+\pi/2)|\big)\vee\big(|\cos(x)|+|\cos(x+\pi/2)|\big).\]
In fact only rotations through multiples of $\pi/2$ are isometries
on $C(S^1)$ for the free norm. Of course, all rotations of $FBL(n)$
will be isomorphisms.

There is an obvious procedure for obtaining a rotation invariant
norm from the free norm, namely to take the average, with respect to
Haar measure on the group of rotations, of the free norms of
rotations of a given element. Although this will certainly not be
the free norm, given that it is derived in a canonical manner from
the free norm we might expect that either it is a familiar norm or
else is of some independent interest. It turns out not to be
familiar. This is again easiest to see in the dual.

If we denote this symmetric free norm by $\|\cdot\|_S$ and its dual
norm by $\|\cdot\|^S$ then we have
\[\|\eta_x\|^S=\|\eta_{x_\pi/2}\|^S=\frac1{2\pi}\int_0^{2\pi}|\sin(t)|\vee|\cos(t)|\;dt=\frac{2\sqrt{2}}{\pi}\]
and
\begin{align*}
\|\eta_x&+\eta_{x_\pi/2}\|^S\\
&=\frac1{2\pi}\int_0^{2\pi}
\big(|\sin(x)|+|\sin(x+\pi/2)|\big)\vee\big(|\cos(x)|+|\cos(x_\pi/2)|\big)\;dt\\
&=\frac{4}{\pi}\end{align*} so that the symmetric free norm is not
an AL-norm, which is the natural symmetric norm on $C(S^1)^*$, nor
an AM-norm. In fact $\|\eta_x+\eta_{x+t}\|^S$ can take any value
between $\frac{4}{\pi}$ and $\frac{4\sqrt{2}}{\pi}$ so the symmetric
free norm cannot be any $L^p$ norm either, implausible though that
would be anyway.

\section{Characterizing the Number of  Generators.}

Apart from wanting to understand how the number of generators affects the Banach lattice structure of $FBL(A)$, we would like to
know when $FBL(A)$ is a classical Banach lattice or has various properties generally considered desirable. The answer to this is ``not very often''!
It turns out that such properties can be used to characterize the number of generators, at least in a rather coarse manner.

In fact several properties that are normally considered ``good'' are only possessed by a free Banach lattice if it has only one generator. We gather several of these into our first result. We know that in the finitely generated case, $FBL(n)$ has a lot of AL-structure as well as some AM-structure. There is another
area of Banach lattice theory where the two structures occur mixed together, namely in injective Banach lattices in the category of Banach lattices and contractive positive operators, see \cite{Ha}. As injective Banach lattices
are certainly Dedekind complete we cannot have $FBL(n)$ being injective if $n>1$. It might be thought possible that $FBL(A)^*$ was injective, but that also
turns out to be false unless $|A|=1$.

\begin{theorem}\label{onegen}If $A$ is a non-empty set then the following are equivalent:
\begin{enumerate}
\item $|A|=1$.
\item $FBL(A)$ is isometrically an AM-space.
\item $FBL(A)$ is isomorphic to an AL-space.
\item Every bounded linear functional on $FBL(A)$ is order continuous.
\item There is a non-zero order continuous linear functional on $FBL(A)$.
\item $FBL(A)^*$ is an injective Banach lattice.
\end{enumerate}
\begin{proof}
If $A$ is a singleton then $\Delta_A=[-1,1]$ and $FBL(A)$ may be identified with $H(\Delta_A)$ which in turn may be identified with $\Real^2$.
The generator is the pair
$g=(-1,1)$. The positive linear functionals $\phi$ such that
$\phi(|g|)\le 1$ are those described by pairs of reals
$(\phi_1,\phi_2)$ with $|\phi_1|+|\phi_2|\le 1$. The free norm that
they induce on $\Real^2$ is precisely the supremum norm.

If $|A|>1$ then by Corollary \ref{ellinftyindual} $FBL(A)^*$ contains an order isometric copy of $\ell_\infty(A)$ so is not an AL-space and therefore $FBL(A)$
is not an AM-space. This establishes that $(i)\Leftrightarrow(ii)$.

It is clear that $(i)\Rightarrow(iii)$ although even in this case it is clear that $FBL(1)$ is not isometrically an
AL-space. $FBL(2)$, on the
other hand is isomorphic to continuous functions on a square so is
certainly not isomorphic to an AL-space. In view of Proposition
\ref{sublattice} and the fact that every closed sublattice of an
AL-space is itself an AL-space we see that $(iii)\Rightarrow(i)$.

It is clear that $(i)\Rightarrow(iv)\Rightarrow(v)$. To show that $(v)\Rightarrow(i)$, suppose that $|A|>1$ and that $\phi$ is a non-zero order continuous linear functional on $FBL(A)$. By continuity of $\phi$ and density of $FVL(A)$ in $FBL(A)$, $\phi_{|FVL(A)}\ne 0$. Similarly, as $FVL(A)=\bigcup\{FVL(F):F\subseteq A, |F|<\infty\}$ we may choose a finite subset $F\subseteq A$ with $\phi_{|FVL(F)}\ne 0$ so certainly $\phi_{|FBL(F)}\ne 0$. Without loss of generality, as long as $|A|>1$ we may assume that $|F|>1$. As $FBL(F)$ is regularly embedded in $FBL(A)$, by Proposition \ref{orderclosedsublattice}, $\phi_{|FBL(F)}$ is order continuous. As a vector lattice, we may identify $FBL(F)$ with $C(S_F)$, where $S_F$ is the 
$\ell_infty$ unit sphere in $\Delta_F$. Certainly $S_F$ is a dense in itself, metrizable (and hence separable) compact Hausdorff space so it follows from Proposition 19.9.4 of \cite{Se} that $\phi_{|FBL(F)}=0$, contradicting our original claim.

Certainly $FBL(1)^*$, being an AL-space, is injective, \cite{Lo} Proposition 3.2. We know from Corollary \ref{dualembedding} that if $|A|>1$ then $FBL(2)^*$ is isometrically order isomorphic to a (projection) band in $FBL(A)^*$. If $FBL(A)^*$ were injective then certainly $FBL(2)^*$ would also be injective. Recall that Proposition 3G of \cite{Ha} tells us that an injective Banach lattice either contains
a sublattice isometric to $\ell_\infty$, or else is  isometrically isomorphic to a finite AM-direct sum of AL-spaces. We know that $FBL(2)$ is order and norm isomorphic to continuous functions on the square $F_2$ so that $FBL(2)^*$ is norm and order isomorphic to the space of measures on $F_2$ so certainly has an order continuous norm. Thus it does not contain even an isomorphic copy of $\ell_\infty$ by Corollary 2.4.3 of \cite{MN}, so it certainly suffices to show that $FBL(2)^*$ cannot be decomposed into
a non-trivial finite AM-direct sum of bands of any nature.

The dual of $FBL(2)$ can be identified, as a vector lattice, with the regular Borel measures on $F_2$.
The dual free norm amounts to $\|\mu\|=\max\{\int |\delta_1|\;d|\mu|, \int |\delta_2|\;d|\mu|\}$, where $\delta_i$ is the projection onto the $i$'th coordinate.
It is clear that
$\int |\delta_1|\;d|\mu|=0$ if and only if $\mu$ is supported by $S_1=\{\langle 0,-1\rangle, \langle 0,1\rangle\}$ whilst
$\int |\delta_2|\;d|\mu|=0$ if and only if $\mu$ is supported by $S_2=\{\langle -1,0\rangle, \langle 1,0\rangle\}$.
If any non-trivial AM-decomposition of $FBL(2)^*$ were possible, into $J\oplus K$ (say), then we can pick $0\ne \mu\in J_+$ and $0\ne \nu\in K_+$. We
may assume that $\|\mu\|=\|\nu\|=1$ and therefore $\|\mu+\nu\|=1$. The fact that $\|\mu\|=\|\nu\|=1$ means that
\[ \int |\delta_1|\;d\mu\vee \int |\delta_2|\;d\mu=\int |\delta_1|\;d\nu\vee \int |\delta_2|\;d\nu=1.\]

Suppose that $\int |\delta_1|\;d\mu=\int |\delta_1|\;d\nu=1$, then we have $1=\|\mu+\nu\|\ge \int |\delta_1|\;d(\mu+\nu)= \int |\delta_1|\;d\mu+\int |\delta_1|\;d\nu=2$,
which is impossible. Similarly, we cannot have $\int |\delta_2|\;d\mu=\int |\delta_2|\;d\nu=1$. If $\int |\delta_1|\;d\mu=\int |\delta_2|\;d\nu=1$ then the fact that
$1=\|\mu+\nu\|\le \int |\delta_1|\;d(\mu+\nu)$ tells us that $\int |\delta_1|\;d\nu=0$ so that $\nu$ is supported by $S_1$. Similarly we see that $\int |\delta_2|\;d\mu=0$
so that $\mu$ is supported by $S_2$. This implies that $FBL(2)^*$ is supported by $S_1\cup S_2$ which is impossible. A similar contradiction arises if
$\int |\delta_2|\;d\mu=\int |\delta_1|\;d\nu=1$.

\end{proof}
\end{theorem}

It is already clear that free Banach lattices on more than one generator are not going to be
amongst the classical Banach lattices. Isomorphism with AM-spaces is
still possible and turns out to determine whether or not the number
of generators is finite.

\begin{theorem}\label{finitegens}If $A$ is any non-empty set then the following are
equivalent:
\begin{enumerate}
\item $A$ is finite.
\item $FBL(A)$ is isomorphic to $H(\Delta_A)$ under the supremum norm.
\item $FBL(A)$ has a strong order unit.
\item $FBL(A)$ is isomorphic to an AM-space.
\item $FBL(A)^*$ has an order continuous norm.
\end{enumerate}
\begin{proof}We have already seen that
(i)$\Rightarrow$(ii)$\Rightarrow$(iii). It is well known and simple to prove that (iii)$\Rightarrow$(iv). That
(iv)$\Rightarrow$(v) is
because the dual of an AM-space is an AL-space which has an order
continuous norm and the fact that order continuity of the norm is
preserved under (not necessarily isometric) isomorphisms. In order
to complete the proof we need only prove that (v)$\Rightarrow$(i).

If $\carda$ is infinite then $FBL(A)^*$ contains a weak$^*$-closed band that is isometrically order isomorphic to $\ell_\infty$, by Corollary \ref{ellinftyindual}. By Theorem 2.4.14 of \cite{MN} this is equivalent to
$FBL(\carda)^*$ \emph{not} having an order continuous norm (and to
many other conditions as well.)
\end{proof}
\end{theorem}

In a similar vein, we can characterize, amongst free Banach
lattices, those with a countable number of generators. Before doing so, though, we note that once there are infinitely many generators then there is an immediate connection between the number of generators and the cardinality of dense subsets. Perhaps not entirely unexpectedly, given Corollary \ref{ctbledisjt}, the same result holds for order intervals. Recall that the \emph{density character} of a topological space is the least cardinal of a dense subset,

\begin{theorem}\label{densitycharacter}If $\carda$ is an infinite cardinal then the following conditions on a set $A$ are equivalent:
\begin{enumerate}
\item $\card(A)=\carda$.
\item $FBL(A)$ has density character $\carda$.
\item The smallest cardinal $\cardb$ such that every order interval in $FBL(A)$ has density character at most $\cardb$ is $\carda$.
\end{enumerate}
\begin{proof}Let $\carda=\card(A)$, $\cardb$ be the density character of $FBL(A)$ and $\cardc$ the smallest cardinal which is at least as large as the density character of every order interval in $FBL(A)$. We need to show that $\carda=\cardb=\cardc$.

The free vector lattice over $\Rational$ with $\carda$ many generators has cardinality precisely $\carda$, given that $\carda$ is infinite. That is dense in $FVL(A)$ and hence in $FBL(A)$ for the free norm, so $\cardb\le \carda$. Clearly $\cardc\le \cardb$.  let $K$ be a compact Hausdorff space such that the smallest cardinality of a dense subset of $C(K)$, and hence of the unit ball in $C(K)$, is $\carda$. For example we could take $K=[0,1]^{\carda}$. There is a bounded lattice homomorphism $T:FBL(A)\to C(K)$ which maps the generators of $A$ onto a dense subset of the unit ball of $C(K)$. The proof of Proposition \ref{freequotients} shows that $T$ is onto. Let $\1_K$ denote the constantly one function on $K$. The order interval $[-T^{-1}\1_K,T^{-1}\1_K]$ has a dense subset of cardinality at most $\cardc$. As $T$ is a surjective lattice homomorphism, $T([-T^{-1}\1_K,T^{-1}\1_K])=[-\1_K,\1_K]$, and this will have a dense subset of cardinality at most $\cardc$. Hence $\carda\le \cardc$. This establishes that $\carda=\cardb=\cardc$.
\end{proof}
\end{theorem}

For the
statement of the next result which characterizes a free Banach lattice having countably many generators, we need to recall some definitions. A
\emph{topological order unit} $e$ of a Banach lattice $E$ is an
element of the positive cone such that the closed order ideal
generated by $e$ is the whole of $E$. These are also referred to as
\emph{quasi-interior points.} Separable Banach lattices always
possess topological order units. The \emph{centre} of $E$, $Z(E)$,
is the space of all linear operators on $E$ lying between two real
multiples of the identity. The centre is termed \emph{topologically
full} if whenever $x,y\in E$ with $0\le x\le y$ here is a sequence
$(T_n)$ in $Z(E)$ with $T_ny\to x$ in norm. If $E$ has a topological
order unit then its centre is topologically full. At the other
extreme there are AM-spaces in which the centre is \emph{trivial},
i.e. it consists only of multiples of the identity.

\begin{theorem}
If $A$ is a non-empty set, then the following are equivalent:
\begin{enumerate}
\item $A$ is finite or countably infinite.
\item $FBL(A)$ is separable.
\item Every order interval in $FBL(A)$ is separable.
\item $FBL(A)$ has a topological order unit.
\item $Z\big(FBL(A)\big)$ is topologically full.
\item $Z\big(FBL(A)\big)$ is non-trivial.
\end{enumerate}
\begin{proof}
If $A$ is finite then it follows from the isomorphism seen in Theorem \ref{finitegens}
that $FBL(A)$, and hence its order intervals, is separable. Combining this observation with the preceding theorem shows that
 (i), (ii) and (iii) are equivalent.

We noted earlier that separable Banach lattices always have a
topological order unit. The fact that Banach lattices with a
topological order unit have a topologically full centre is also
widely known, but finding a complete proof in the literature is not
easy. The earliest is in Example 1 of \cite{Or}, but that proof is
more complicated than it need be. A simpler version is in
Proposition 1.1 of \cite{Wi2} and see also Lemma 1 of \cite{dPa}.

Even if $\carda=1$, $FBL(\carda)$ is not one-dimensional so that if
the centre is topologically full then it cannot be trivial.

We know from Proposition \ref{fnrepn} that we may identify
$FBL(A)$ with a sublattice of $H(\Delta_A)$. It is clear, as it contains the coordinate
projections, that it separates points of $\Delta_A$. If $|A|$
is uncountable then $\{0\}$ is not a
$G_\delta$ subset of $\Delta_A$. It follows from Theorem 3.1 of
\cite{Wi} that the centre of this sublattice, and therefore of
$FBL(\carda)$, is then trivial.
\end{proof}
\end{theorem}

\begin{corollary}
 If $A$ is an uncountable set then $FBL(A)$ has trivial centre.
\end{corollary}

Note that this would seem to be the first ``natural'' example of a
Banach lattice with a trivial centre. If $\carda>1$ then $FVL(A)$ always has trivial centre. The details are left to the interested reader.

\section{Lifting Disjoint Families in Quotient Banach Lattices.}

In \cite{We}, Weinberg asked what were the projective objects in the category of abelian $\ell$-groups, pointing out, for example, that
a summand of a free $\ell$-group was projective. Topping studied projective vector lattices in \cite{To} but the reader should be warned that Theorem 8, claiming that
countable positive disjoint families in quotients $L/J$ of vector lattices $L$ lift to positive disjoint families in $L$, is false. In fact that is only possible for
an Archimedean Riesz space if the space is a direct sum of
copies of the reals, see \cite{Co} and \cite{Mo}.

 Later in this paper we will study projective Banach
lattices, which are intimately connected with quotient spaces. We
will need to know when disjoint families in a quotient Banach lattice $X/J$
can be lifted to disjoint families in $X$. As this is a question of
considerable interest in its own right and also because the results
that we need do not seem to be in the literature already, we present
them in a separate section here.

It is well-known, although we know of no explicit reference, that
any finite disjoint family $(y_k)_{k=1}^n$ in a quotient Riesz space
$X/J$ can be lifted to a disjoint family $(x_k)_{k=1}^n$ in $X$ with
$Qx_k=y_k$, where $Q:X\to X/J$ is the quotient map. If we restrict
attention to norm closed ideals in Banach lattices then, unlike the vector lattice case, we can
handle countably infinite disjoint liftings, but not larger ones.
This does not contradict the vector lattice result cited above as
there are many non-closed ideals in a Banach lattice.

\begin{theorem}\label{countablelifting}
If $X$ is a Banach lattice, $J$ a closed ideal in $X$, $Q:X\to X/J$
the quotient map and $(y_k)_{k=1}^\infty$ is a disjoint sequence in
$X/J$ then there is a disjoint sequence $(x_k)$ in $X$ with
$Qx_k=y_k$ for all $k\in \Natural$.
\begin{proof}
It suffices to consider the case
that each $y_n\ge 0$ and $\|\sum_{k=1}^\infty \|y_k\|<\infty$. Define $z_n=\sum_{k=n}^\infty y_k\in X/J$ and note that $z_{n+1}$ is disjoint from 
$y_1,\dots, y_n$. The sequence $(x_n)$ will be constructed inductively. 

For $n=1$ we start by choosing $\tilde{x}_1, \tilde{u}_1\in X_+$ with $Q\tilde{x}_1=y_1$ and $Q\tilde{u}_1=z_2$. Now define 
$x_1=\tilde{x}_1-\tilde{x}_1\wedge \tilde{u}_1$ and $u_1=\tilde{u}_1-\tilde{x}_1\wedge \tilde{u}_1$. Then ${x_1,u_1}$ is a disjoint system and, as 
$Q(\tilde{x}_1\wedge \tilde{u}_1)=y_1\wedge z_2=0$, $Qx_1=y_1$ and $Qu_1=z_2$.

Now suppose that we have constructed a disjoint system $\{x_1, \dots, x_n, u_n\}$ with $Qx_j=y_j (1\le j\le n)$ and $Qu_n=z_{n+1}$. 
As $0\le y_{n+1}, z_{n+2}\le z_{n+1}$, There  are $\tilde{x}_{n+1}, \tilde{u}_{n+1}\in X_+$ with $\tilde{x}_{n+1},\tilde{u}_{n+1}\le u_n$, 
$Q\tilde{x}_{n+1}=y_{n+1}$ and $Q\tilde{u}_{n+1}=z_{n+2}$. Let $x_{n+1}=\tilde{x}_{n+1}-\tilde{x}_{n+1}\wedge \tilde{u}_{n+1}$ and 
$u_{n+1}=\tilde{u}_{n+1}-\tilde{x}_{n+1}\wedge \tilde{u}_{n+1}$, then $x_{n+1}\wedge u_{n+1}=0, Qx_{n+1}=y_{n+1}$ and $Qu_{n+1}=z_{n+2}$. Moreover, 
$0\le x_{n+1}, u_{n+1}\le u_n$ so that $\{x_1,\dots, x_n, x_{n+1}, u_{n+1}\}$ is again a disjoint system.
\end{proof}
\end{theorem}

The same proof will show, as there is no need to worry about convergence, that:

\begin{proposition}\label{finitelifting}
For any vector lattice $X$ and ideal $J$ in $X$, if $Q:Xto X/J$ is the quotient map then for any disjoint family $(y_k)_{k=1}^n$ in $X/J$ there is a disjoint family
 $(x_k)_{k=1}^n$ in $X$ with $Qx_k=y_k$ for $1\le k\le n$.
\end{proposition}

Even in Banach lattices, Theorem \ref{countablelifting} is as far as we can go.

\begin{example}\label{nolifting}
Given any uncountable disjoint family in a Banach lattice $X$, we know from Proposition \ref{freequotients} that there is a free Banach lattice $FBL(\carda)$ and a closed ideal $J$ in $FBL(\carda)$ such that $X$ is isometrically order isomorphic to $FBL(\carda)/J$. As a disjoint family in a free Banach lattice has to be countable, Corollary \ref{ctbledisjt}, the disjoint family cannot possibly be lifted to $FBL(\carda)$.

A slightly more concrete example may be found using Problem 6S of \cite{GJ} where it is shown that
$\beta\Natural\setminus\Natural$ contains continuum many disjoint
non-empty open and closed subsets. I.e. $\ell_\infty/c_0$ contains
continuum many non-zero disjoint positive elements. As $\ell_\infty$
contains only countably many disjoint elements, we cannot possibly
lift each of this continuum of disjoint elements in
$\ell_\infty/c_0$ to disjoint elements in $\ell_\infty$. The same
will be true of any uncountable subset of these disjoint positive
elements of $\ell_\infty/c_0$, so this shows that lifting of disjoint positive
families of cardinality $\aleph_1$ is not possible.
\end{example}

An apparently simpler problem is to start with two subsets $A$ and
$B$ in $X/J$ with $A\perp B$ and seek subsets $A'$, $B'$ of $X$ with
$A'\perp B'$, $Q(A')=A$ and $Q(B')=B$. Again countability is vital
to the success of this attempt, in fact it allows us to do much
more.

\begin{proposition}If $X$ is a Banach lattice, $J$ a closed ideal in $X$, $Q:X\to X/J$
the quotient map and $(A_n)_{n=1}^\infty$ be a sequence of
\emph{countable} subsets of $X/J$ with $A_m\perp A_n$ if $m\ne n$
then there are subsets $(B_n)$ of $X$, with $B_m\perp B_n$ if $m\ne
n$ and $Q(B_n)=A_n$ for each $n\in\Natural$.
\begin{proof}
As above, there is no loss of generality in assuming that each
$A_n\subset (X/J)_+$. Enumerate each set as
$A_n=\{a^n_k:k\in\Natural\}$ (there is no difference, apart from
notation, if one or both set is finite). Let $v_n=\sum_{k=1}^\infty
a^n_k/(2^k\|a_k\|)$ so that $v_m\perp v_n$ if $m\ne n$ and $0\le
a^n_k\le 2^k \|a_k\| v_n$ for $k,n\in\Natural$. We know from Theorem
\ref{countablelifting} that there is a disjoint sequence $(u_n)$ in
$X_+$ with $Q(u_n)=v_n$. For any $a^n_k\in A_n$ we can find
$c^n_k\in X_+$ with $Q(c^n_k)=a^n_k$. Now set $b^n_k=c^n_k\wedge
(2^k \|a_k\| u_n)$ so that we still have
\[Q(b^n_k)=Q(c^n_k)\wedge\big(2^k \|a_k\| Q(u_n)\big)=a^n_k\wedge
(2^k \|a_k\| v_n)=a^n_k.\] Also each $b^n_k\in u_n^{\perp\perp}$ so
that if $m\ne n$ then for any choice of $j$ and $k$ we see that
$b^m_j\perp b^n_k$ as $u_m\perp u_n$. Now defining
$B_n=\{b^n_k:k\in\Natural\}$ gives the required sets.
\end{proof}
\end{proposition}

Considering the case of singleton sets, the example above shows that
we cannot allow an uncountable number of disjoint families. Nor can
we allow even one of the families to be uncountable.

In the case that $X=C(K)$, for $K$ a compact Hausdorff space, a
closed ideal $J$ is of the form $F=\{f\in C(K):f_{|A}\equiv 0\}$ for
some closed subset $A\subset K$ and the quotient $X/J$ may be
identified with $C(A)$ in the obvious manner. For two elements
$f,g\in C(K)$, $f\perp g$ if and only if the two sets
$f^{-1}(\Real\setminus\{0\})$ and $g^{-1}(\Real\setminus\{0\})$ are
disjoint.

\begin{example}\label{TPex}The \emph{Tychonoff plank} is the topological
space $[0,\omega]\times[0,\omega_1]\setminus\{(\omega,\omega_1)\}$
where $\omega$ is the first infinite ordinal and $\omega_1$ the
first uncountable ordinal. This is renowned as an example of a
non-normal Hausdorff space. The sets
$U=[0,\omega)\times\{\omega_1\}$ and $V=\{\omega\}\times
[0,\omega_1)$ are disjoint closed subsets which cannot be separated
by disjoint open sets. See for example \S8.20 of \cite{GJ}. If we
add back in the removed corner point, and define $A=U\cup
V\cup\{(\omega,\omega_1)\}$ then $U$ and $V$ become open subsets of
$A$. Any disjoint open subsets of the whole product space which
intersected $A$ in $U$ and $V$ respectively would, with the corner
point removed if necessary, separate the closed sets $U$ and $V$ in
the plank. This contradiction shows that the lifting is not
possible.

Each point of $U$ is isolated so their characteristic functions lie
in $C(A)$ giving a (countable) family $F$ with
$U=\bigcup\{f^{-1}(\Real\setminus\{0\}):f\in F\}$. Let $G$ be a
family of functions in $C(A)$ such that
$V=\bigcup\{g^{-1}(\Real\setminus\{0\}):g\in G\}$, which is
certainly possible using Urysohn's lemma.   If these could be lifted
to disjoint families $L$ and $M$ then
$\bigcup\{f^{-1}(\Real\setminus\{0\}):f\in L\}$ and
$\bigcup\{f^{-1}(\Real\setminus\{0\}):f\in M\}$ would be disjoint
open subsets of $K$ which intersected $A$ in disjoint open sets
which equalled, respectively $U$ and $V$, which we have seen is
impossible.
\end{example}

\section{Projective Banach Lattices.}

\begin{definition}
A Banach lattice $P$ is \emph{projective} if whenever $X$ is a
Banach lattice, $J$ a closed ideal in $X$ and $Q:X\to X/J$ the
quotient map then for every linear lattice homomorphism $T:P\to X/J$
and $\epsilon>0$ there is a linear lattice homomorphism
$\hat{T}:P\to X$ such that
\begin{enumerate}
\item $T=Q\circ \hat{T}$,
\item $\|\hat{T}\|\le (1+\epsilon)\|T\|$.
\end{enumerate}
\end{definition}

Even if we take $P=\Real$, which is easily seen to be projective
given this definition, it is clear that we cannot replace
$1+\epsilon$ by $1$ as the quotient norm is an infimum which need
not be attained. There are projective Banach lattices, because:

\begin{proposition}\label{freeisprojective}A free Banach lattice is projective.
\begin{proof}
Let $(\delta_a)_{a\in\carda}$ be the generators of the free Banach
lattice $F$. Suppose that $X$ is a Banach lattice, $J$ a closed
ideal in $X$, $Q;X\to X/J$ the quotient map, $T:F\to X/J$ a lattice
homomorphism and $\epsilon>0$. For each $\alpha\in\carda$, there is
$x_a\in X$ with $Qx_a=T\delta_a$ and $\|x_a\|\le
(1+\epsilon)\|T\delta_a\|\le (1+\epsilon)\|T\|$, using the
definition of the quotient norm. As $F$ is free there is a linear
lattice homomorphism $\hat{T}:F\to X$ with $\hat{T}\delta_a=x_a$ for
all $a\in \carda$ and $\|\hat{T}\|\le \sup\{\|x_a\|:a\in\carda\}\le
(1+\epsilon)\|T\|$. As $(Q\circ \hat{T})\delta_a=T\delta_a$ for all
$a\in\carda$ and both $Q\circ\hat{T}$ and $T$ are linear lattice
homomorphisms they must coincide on the vector lattice generated by
the $\delta_a$ and, by continuity, on $F$.
\end{proof}
\end{proposition}

We can characterize projective Banach lattices in a reasonably
familiar manner.

\begin{theorem}\label{projs}
The following conditions on a Banach lattice $P$ are equivalent.
\begin{enumerate}
\item $P$ is projective.
\item For all $\epsilon>0$ there are:
\begin{enumerate}
\item a free Banach lattice $F$,
\item a closed sublattice $H$ of $F$ and a lattice isomorphism
$\I:H\to P$ with $\|\I\|,\|\I^{-1}\|\le 1+\epsilon$, and
\item a lattice homomorphism projection $R:F\to H$ with $\|R\|\le
1+\epsilon$.
\end{enumerate}
\item For all $\epsilon>0$ there are:
\begin{enumerate}
\item a projective Banach lattice $F$,
\item a closed sublattice $H$ of $F$ and a lattice isomorphism
$\I:H\to P$ with $\|\I\|,\|\I^{-1}\|\le 1+\epsilon$, and
\item a lattice homomorphism projection $R:F\to H$ with $\|R\|\le
1+\epsilon$.
\end{enumerate}
\end{enumerate}
\begin{proof}
To see that (i)$\Rightarrow$(ii), suppose that $P$ is projective, let
$F$ be a free Banach lattice and $J$ a closed ideal in $F$ such that
$P$ is isometrically order isomorphic to the quotient $F/J$
\emph{via} the linear lattice isomorphism $\I:P\to F/J$, which is
always possible using Proposition \ref{freequotients}. Let $Q:F\to
F/J$ be the quotient map. As $P$ is projective, for any $\epsilon>0$
there is a linear lattice homomorphism $\hat{\I}:P\to F$ with
$Q\circ \hat{\I}=\I$ and $\|\hat{\I}\|\le
(1+\epsilon)\|\I\|=1+\epsilon$. As $Q\circ \hat{\I}$ is injective,
$\hat{\I}$ is also injective and $\hat{\I}P$ is a closed sublattice
of $F$ as $\|\hat{\I}p\|\ge \|Q(\hat{\I}p\|=\|\I p\|=\|p\|$. The map
$\hat{\I}\circ \I^{-1}\circ Q$ is a lattice homomorphism which
projects $F$ onto $\hat{\I}(P)$ and $\|\hat{\I}\circ \I^{-1}\circ
Q\|\le \|\hat{\I}\|\le 1+\epsilon$, so (2)(b) holds. We know that
$\|\hat{\I}\|\le 1+\epsilon$ and $\hat{\I}^{-1}=\I^{-1}\circ Q$ so
that $\|\hat{\I}^{-1}\|=1$ and (2)(c) holds.

In view of Proposition \ref{freeisprojective}, clearly
(ii)$\Rightarrow$(iii).

 Suppose that (iii) holds, and in particular that (a), (b) and (c) hold for the
real number $\eta$. Suppose that $X$ is any Banach lattice, $J$ a
closed ideal in $X$, $Q:X\to X/J$ the quotient map, $\eta>0$ and
that $T:P\to X/J$ is a linear lattice homomorphism. The map
$T\circ\I\circ R:F\to X/J$ is also a linear lattice homomorphism
with $\|T\circ \I\circ R\|\le \|T\|\|\I\|\|R\|\le (1+\eta)^2 \|T\|$.
As $F$ is projective there is a linear lattice homomorphism $S:F\to
X$ with $Q\circ S=T\circ \I\circ R$ and $\|S\|\le (1+\eta)\|T\circ
\I\circ R\|\le (1+\eta)^3\|T\|$. Now let $\hat{T}=S\circ
\I^{-1}:P\to X$, which is also a linear lattice homomorphism, so
that $\|\hat{T}\|\le \|S\|\|\I^{-1}\|\le (1+\eta)^4\|T\|$ and
\[Q\circ\hat{T}=Q\circ(S\circ \I^{-1})=(Q\circ S)\circ
\I^{-1}=(T\circ \I\circ R)\circ \I^{-1}=T.\] By choosing $\eta$
small enough we can ensure that $(1+\eta)^4\le1+\epsilon$ and we
have shown that $P$ is projective.
\end{proof}
\end{theorem}

In particular, in light of Corollary \ref{sepfreequotients},  all the
separable projective Banach lattices that we produce later will
(almost) embed in $FBL(\aleph_0)$ reinforcing the richness of its
structure.

Combining Theorem \ref{projs} with Corollary \ref{freefinal} we have:

\begin{corollary}
The real-valued lattice homomorphisms on a projective Banach lattice separate points.
\end{corollary}

In particular this tells us that, for finite $p$, the Banach lattice $L_p([0,1])$ is \emph{not} projective.

Similarly, from Corollary \ref{ctbledisjt} and Theorem \ref{projs}, using the lattice homomorphism projection from a free Banach lattice onto a projective, we see:

\begin{corollary}\label{projdisjt}Every disjoint system in a projective Banach lattice is at most countable.
\end{corollary}

Although, in a sense, Theorem \ref{projs} gives a complete description of
projective Banach lattices, given that we know little about free
Banach lattices it actually tells us very little. One immediate
consequence, given that $FBL(1)$ may be identified with
$\ell_\infty(2)$, is:

\begin{corollary}
The one dimensional Banach lattice $\Real$ is projective.
\end{corollary}

Of course, this is easy to verify directly, but it does show that
there are projective Banach lattices which are not free.

Let us note also one rather simple consequence of the characterization of projectives in Theorem \ref{projs}.

\begin{corollary}If $X$ is a projective Banach lattice, $H$ a closed sublattice of $X$ for which there is a contractive lattice homomorphism projecting $X$ onto $H$, then $H$ is a projective Banach lattice.
\end{corollary}

\section{Which Banach lattices are projective?}

We will now approach matters from the other end. We try to find out
as much as we can about projective Banach lattices and deduce
information about the structure of free Banach lattices. We will
start by identifying some ``small'' Banach lattices, apart from free
ones, which are projective. After that we will show that certain
AL-sums of projectives are again projective.

Our first positive result may be slightly surprising, given that
when dealing with Banach spaces the free and projective objects are precisely the
spaces $\ell_1(I)$, \cite{Se}, Theorem 27.4.2.

\begin{theorem}Every finite dimensional Banach lattice is
projective.
\begin{proof}Let $P$ be a finite dimensional Banach lattice, $X$ an
arbitrary Banach lattice, $J$ a closed ideal in $X$, $Q:X\to X/J$
the quotient map, $T:P\to X/J$ a lattice homomorphism and
$1\ge\epsilon>0$. We identify $P$ with $\Real^n$ with the pointwise
order and normed by some lattice norm $\|\cdot\|_P$. Without loss of
generality we may assume that the standard basic vectors in
$\Real^n$, $e_k$, all have $\|e_k\|_P=1$. Let $\{p_k:1\le k\le m\}$
be an $\epsilon$-net for the compact set $\{p\in
\Real^n_+:\|p\|_P=1\}$. We write $p_k=(p_k^1,p_k^2,\dots,p_k^n)$.

As $T$ is a lattice homomorphism, the family $(Te_k)_{k=1}^n$ is a
disjoint family in $(X/J)_+$ so by Proposition \ref{finitelifting}
there is a disjoint family $(s_k)_{k=1}^n$ in $X_+$ with $Qs_k=Te_k$
for $1\le k\le n$. By the definition of the quotient norm, for each
$k$ there is $t_k\in X$ with $Qt_k=Te_k$ and $\|t_k\|\le
\|Te_k\|+\epsilon\le \|T\|+\epsilon$. Now, let $x_k= s_k\wedge
t_k^+$, so that the family $(x_k)$ remains disjoint. As $Q$ is a
lattice homomorphism, $Qx_k=Qs_k\wedge Qt_k^+=(Te_k)\wedge
(Te_k)^+=Te_k$. Also, we now have $\|x_k\|\le \|t_k^+\|\le
\|t_k\|\le \|T\|+\epsilon$.

Also, for each $i\in \{1,2,\dots,m\}$ there is $q_i\in X_+$ with
$Qq_i=Tp_i$ and $\|q_i\|\le \|Tp_i\|+\epsilon\le \|T\|+\epsilon$.

Define $z_k=x_k\wedge {\bigwedge'}_{i=1}^m (p_i^k)^{-1} q_i$ where
the ${}'$ indicates that terms where $p_i^k=0$ are omitted. As the
family $(x_k)$ is disjoint, the same is true for the family $(z_k)$.
If $p_i^k>0$ then $(p_i^k)^{-1}p_i\ge e_k$ so that $(p_i^k)^{-1}
Qq_i=(p_i^k)^{-1}Tp_i\ge Te_k$ so that $Qz_k=Qx_k=Te_k$.

Define $Se_k=z_k$ and extend $S$ linearly to a lattice homomorphism
(because the $(z_k)$ are disjoint) of $\Real^n\to X$. Clearly
$Q\circ S_k=T$. As $\Real^n$ is finite dimensional, there is a
constant $K\in\Real_+$ such that $\|x\|_1\le K\|x\|_P$ for all
$x\in\Real^n$. It follows that
\begin{align*}
\left\|S(\sum_{k=1}^n \lambda_k e_k\right\|&\le \sum_{k=1}^n
|\lambda_k|
\|Se_k\|\\
&=\sum_{k=1}^n |\lambda_k| \|z_k\|\\
&\le \sum_{k=1}^n |\lambda_k|\|x_k\|\\
&\le \left\|\sum_{k=1}^n \lambda_k e_k\right\|(\|T\|+1)\\
&\le K(\|T\|+1)\left\|\sum_{k=1}^n \lambda_k e_k\right\|_P
\end{align*}
so that $\|S\|\le K(\|T\|+1)$. Note that this estimate is
independent of the choice of $\epsilon$.

In order better to estimate the norm of $S$, we write
$p_i=\sum_{k=1}^n p_i^k e_k$ and see that
\begin{align*}Sp_i&=\sum_{k=1}^n S(p_i^k e_k)\\&=\sum_{k=1}^n p_i^k
Se_k\\&=\sum_{k=1}^n p_i^k z_k.
\end{align*}
Also, if $p_i^k=0$ then certainly $p_i^k z_k\le q_i$, whilst if
$p_i^k>0$ then  $p_i^k z_k\le p_i^k(p_i^k)^{-1}q_i=q_i$. As $p_i^j
z_j\perp p_i^k z_k$ if $j\ne k$ we see that $\sum_{k=1}^n p_i^k
z_k\le q_i$ so that $Sp_i\le q_i$ and $\|Sp_i\|\le \|q_i\|\le
\|T\|+\epsilon$. Now if we take an arbitrary $p\in \{P_+:\|p\|=1\}$
then we can choose $i$ with $\|p-p_i\|_P<\epsilon$, so that
\begin{align*}
\|Sp\|&\le \|Sp_i\|+\|S\|\|p-p_i\|_P\\
&\le\|T\|+\epsilon+K(\|T\|+1)\epsilon
\end{align*}
which can be made as close to $\|T\|$ as we desire.
\end{proof}
\end{theorem}

The spaces $C(K)$, for $K$ a compact Hausdorff space, play a
distinguished r\^{o}le in the general theory of Banach lattices so
it is worth knowing which $C(K)$ spaces are projective. We give here
a partial answer, which is already of substantial interest. We refer
the reader to \cite{Bo} for basic concepts about retracts.

\begin{theorem}\label{fdCKs}If $K$ is a compact subset of $\Real^n$ for some
$n\in\Natural$ then the following are equivalent:
\begin{enumerate}
\item $C(K)$ is a projective Banach lattice under some norm.
\item $C(K)$ is projective under the supremum norm.
\item $K$ is a neighbourhood retract of $\Real^n$.
\end{enumerate}
\begin{proof}
Without loss of generality we may suppose that $K$ is a subset of
the unit ball in $\Real^n$ for the supremum norm. We write $p_k$ for
the restriction to $K$ of the $k$'th coordinate projection in
$\Real^n$ and $p_0$ for the constantly one function on $K$. The
vector sublattice generated by the $\{p_k:0\le k\le n\}$ is
certainly dense in $C(K)$ by the Stone-Weierstrass theorem. As
$F(n+1)$ is free there is a bounded vector lattice homomorphism
$T:F(n+1)\to C(K)$ with $T(\delta_k)=p_{n-1}$. We know that,
algebraically, we may identify $F(n+1)$ with $C(F_{n+1})$ and that
the constantly one function on $F_{n+1}$ is precisely
$\bigvee_{k=1}^{n+1}|\delta_k|$. As $\bigvee_{k=0}^n|p_k|=p_0$, here
is where we use the boundedness assumption on $K$, we may regard $T$
as a unital lattice homomorphism from $C(F_{n+1})$ to $C(K)$. Such
maps are of the form $f\mapsto f\circ \phi$ where $\phi:K\to
F_{n+1}$ is continuous. The image of $C(F_{n+1})$ is dense in $C(K)$
and it is well known that the image of such composition maps is
closed so that $T$ is onto. This is equivalent to $\phi$ being injective.
I.e. we have a topological embedding of $K$ into $F_{n+1}$ and we
may regard $T$ as simply being the restriction map from $C(F_{n+1})$
to $C(K)$. So far we have not used the assumption that $C(K)$ is
projective.

If $J$ is the kernel of $T$ then $C(F_{n+1})/J$ is isomorphic to
$C(K)$. If $C(K)$ is projective (even in a purely algebraic sense)
then there is a vector lattice homomorphism $U:C(K)\to C(F_{n+1})$
with $Uf_{|K}=f$ for all $f\in C(K)$. But $U$ is of the form
\[Uf(p)=\begin{cases}
w(p) f(\pi p)&(p\in U)\\
0&(p\notin U) \end{cases}
\]
where $w$ is a non-negative continuous real-valued function on
$F_{n+1}$ and $\pi:F_{n+1}\setminus w^{-1}(0)\to K$, so we must have
$w(p)=1$ and $\pi p=p$ for $p\in K$. Thus $F_{n+1}\setminus
w^{-1}(0)$ is open and contains $K$ so that $\pi$ is a neighbourhood
retract of $F_{n+1}$ onto $K$. If we remove any single point from
$F_{n+1}$ that is not in $K$ then what remains is homeomorphic to
$\Real^n$ so we have a neighbourhood retraction from $\Real^n$ onto
$K$. This only fails to be possible if $K=F_{n+1}$, and that is not
homeomorphic to a subset of $\Real^n$ by the Borsuk-Ulam Theorem,
see for example Theorem 5.8.9 of \cite{Sp}.  Thus (i) implies (iii).

Clearly (ii) implies (i), so we need only prove that (iii) implies (ii).
The blanket assumption on $K$ tells us that it is homeomorphic to a
subset of one face $G$ of $F_{n+1}$. By scaling it if necessary, we
may assume that it is a neighbourhood retract of $G$  and therefore
of the whole of $F_{n+1}$. That allows us to construct a continuous
$w:F_{n+1}\to [0,1]$ with $U=\{p\in F_{n+1}:w(p)>0\}\subset G$,
$K\subset w^{-1}(1)$ and a continuous retract  $\pi:U\to K$. The
vector lattice homomorphism $U:C(F_{n+1})\to C(F_{n+1})$ defined by
\[Uf(p)=\begin{cases}
w(p) f(\pi p)&(p\in U)\\
0&(p\notin U) \end{cases}
\] is certainly a projection. For any $p\in
F_{n+1}$ we have, writing $J^K=\{f\in C(F_{n+1}):f_{|K}\equiv0\}$,
\begin{align*}\|Uf\|_F&=\|Uf\|_\infty&\text{(Corollary
\ref{AMsubset})}\\
&=\sup\{|w(p) f(\pi p)|:p\in U\}\\
&\le \sup\{|f(\pi p):p\in U\}\\
&=\le \sup\{|f(k):k\in K\}=\|f_{|K}\|_\infty\\
&=\|f+J^K\|&\text{(Corollary \ref{nicequotients})}\\
&\le \|f\|_F
\end{align*}
so that $U$ is a contraction.

We claim also that the image $UC(F_{n+1})$ is isometrically order
isomorphic to $C(K)$ under its supremum norm. To prove this, it
suffices to prove that $Uf\mapsto F_{|K}$ is an isometry for the
free norm on $Uf$, which is equal to its supremum norm, and the
supremum norm on $f_{|K}$. The calculation above shows that
$\|Uf\|_\infty\le \|f_{|K}\|_\infty$. We also have, for $p\in U$,
$|Uf(p)|=|w(p)||f(\pi p)|\le \|f_{|K}\|_\infty$ as $|w(p)|\le 1$ and
$\pi p\in K$. Thus $\|Uf\|_\infty\le \|f_{|K}\|_\infty$ and we have
our desired isometry.

In view of Theorem \ref{projs}, this shows that $C(K)$ is
projective.
\end{proof}
\end{theorem}

The reader will notice that the first implication would actually
work for an isomorphic version of projectivity. We allude further to
this in \S\ref{problems}.

Notice that some $C(K)$-spaces can be projective for different
(necessarily equivalent) Banach lattice norms. E.g. $C(F_n)$ will be
projective both under the free and supremum norms.

Descriptions of absolute neighbourhood retracts in the
category of compact metric spaces may be found in Chapter V of \cite{Bo}. We note two
particular properties that they have. Firstly, absolute
neighbourhood retracts have only finitely many components
(\cite{Bo}, V.2.7) and if $K$ is an absolute neighbourhood retract
subset of $\Real^n$ then $\Real^n\setminus K$ has only finitely many
components (\cite{Bo}, V.2.20).

In particular, we have

\begin{corollary}The sequence space $c$ is not projective.
\begin{proof}We can identify $c$ with
$C(K_0)$ where $K_0=\{\frac1n:n\in\Natural\}\cup\{0\})$. As
$K_0\subset \Real$ and $K_0$ has infinitely many components it is
not an absolute neighbourhood retract.
\end{proof}
\end{corollary}

There seems little hope of removing the assumption of finite dimensionality from $K$ in Theorem \ref{fdCKs}. We can rescue one implication when we recall that, as closed bounded convex subsets of $\Real^n$ are
absolute retracts in the category of compact Hausdorff spaces, any
compact neighborhood retract of $\Real^n$ will necessarily be an
absolute neighbourhood retract in the category of compact Hausdorff
spaces.

\begin{proposition}If $C(K)$ is a projective Banach lattice under the supremum, or an equivalent, norm then $K$ is an absolute neighbourhood retract in the category of compact Hausdorff spaces.
\begin{proof}
Suppose that $K$ is a closed subset of a compact Hausdorff space $X$. We need to show that there is a continuous retraction $\pi$ of $U$ onto $K$, where $U$ is an open subset of $X$ with $K\subset U$.

The restriction map $R:C(X)\to C(K)$ may be identified with the canonical quotient map of $C(X)$ onto $C(X)/J$ where $J$ is the closed ideal $\{f\in X(K):f_{|K}\equiv 0\}$. If $C(K)$ is projective then the identity on $C(K)$ lifts to a lattice homomorphism $T:C(K)\to C(X)$ with $R\circ T-I_{C(K)}$. There is a continuous function $w$ from $X$ into $\Real_+$ and a continuous map $\pi:U=\{x\in X: w(x)>0\}\to K$ such that
\[Tf(x)=\begin{cases}w(x) f(\pi x)&[w(x)>0]\\
0&[w(x)=0].
\end{cases}\]
If  $k\in K$ then $Tf(k)=f(k)$ so that  $\pi k=k$ and $w(k)=1$ showing that $K\subset U$ and that $\pi$ is a retraction of the open set $U$ onto $K$.
\end{proof}
\end{proposition}

Without knowledge of the properties of absolute neighbourhood retracts in the category of compact Hausdorff spaces, this does not tell us a lot. There seems to be very little material in the literature on absolute neighbourhood retracts in this setting, so we make our own modest contribution here.

\begin{lemma}If $C$ is a compact convex subset of a locally convex space, $K$ a closed subset of $C$ and $U$ an open subset of $C$ with $K\subseteq U\subset C$ then there is an open set $V$ with $K\subseteq V \subseteq U\subseteq C$ such that $V$ has finitely many components.
\begin{proof}
As $U$ is open, if $k\in K$ there is a convex open set $W_k$ with $k\in W_k\subseteq U$, using local convexity. The open sets $W_k$, for $k\in K$, cover the compact set $K$ so there is a finite subcover, $W_1, W_2,\dots, W_n$. Take $V=\bigcup_{k=1}^n W_k$.
\end{proof}
\end{lemma}

\begin{proposition}
If $K$ is an absolute neighbourhood retract in the category of compact Hausdorff spaces then $K$ has only finitely many components.
\begin{proof}
Let $C=P(K)$, the space of probability measures on $K$, with the weak$^*$ topology induced by $C(K)$, which is a locally convex topology under which $C$ is compact as well as certainly being convex. The mapping which takes $k$ to the point mass at $k$ is a homeomorphism of $K$ onto the set of extreme points of $C$. If $K$ is an absolute neighbourhood retract then there is a retraction $\pi:U\to K$ where $U$ is an open subset of $C$ with $K\subseteq U$. By the preceding lemma, there is an open set $V$, with finitely many components, such that $K\subseteq V\subseteq U$. The image of each component of $V$ under $\pi$ is connected and their union is $K$, so that $K$ has only finitely many components.
\end{proof}
\end{proposition}

Thus if $C(K)$ is a projective Banach lattice under any norm then $K$ has only finitely many components. In particular:

\begin{corollary}The sequence space $\ell_\infty$ is not a projective Banach lattice.
\end{corollary}

In \cite{Ba} Baker characterized  projective vector lattices with
$n$ generators as being quotients of $FVL(n)$ by a principal ideal.
If we embed $K_0$ into one of the faces of $F_2$ then we know that
$c$ is isometrically order isomorphic to $FBL(2)/J^{K_0}$. It is
clear that $J^{K_0}$ is a principal closed ideal of $FBL(2)$ and
that $c$ has two generators as a Banach lattice, so the natural
analogue of Baker's result fails in the Banach lattice setting.

The obvious candidate for a projective Banach lattice, as in the
Banach space case, is $\ell_1(I)$ for an arbitrary index set $I$, however Corollary \ref{projdisjt} tells us that if $I$ is an uncountable index set then
$\ell_1(I)$ is definitely \emph{not} a projective Banach lattice. Similarly $\ell_p(I)$ ($1\le p<\infty$) and
$c_0(I)$ are not projective if $I$ is uncountable.

Given that we can lift disjoint sequences it is not difficult to
show that $\ell_1$ is projective. In fact we can show much more.

\begin{theorem}\label{ellonesum}If, for each $n\in\Natural$, $P_n$ is a projective
Banach lattice with a topological order unit then the countable sum $\ell_1(P_n)$,
 under the coordinate-wise order and normed by $\|(p_n)\|_1=
\sum_{n=1}^\infty\|p_n\|$,
is a projective Banach lattice.
\begin{proof}
Let $e_n$ be a topological order unit for $P_n$. We will identify
$P_n$ with the subspace of $\ell_1(P_n)$ in which all entries
apart from the $n$'th are zero and $e_n$ with the corresponding
member of that subspace so that the $e_n$ are all disjoint. If $X$
is a Banach lattice, $J$ a closed ideal in $X$, $Q:X\to X/J$ the
quotient map, $T:\oplus_1(P_n)\to X/J$ a lattice homomorphism and
$\epsilon>0$ then we start by noting that the $Te_n$ are disjoint,
so by Theorem \ref{countablelifting} we can find disjoint $u_n$ in
$X_+$ with $Qu_n=Te_n$. If we write $X_n$ for the closed ideal in
$X$ generated by $u_n$ then the family $(X_n)$ is disjoint in $X$.

Note that the natural embedding of $X_n/(J\cap X_n)$ into $X/J$ is
an isometry onto an ideal and that $T(P_n)\subset X_n/(J\cap X_n)$ as $e_n$ is a topological order unit for $P_n$ and
$T$ is a lattice homomorphism. The projectivity of $P_n$ allows us to
lift $T_n$ to a lattice homomorphism $\hat{T_n}:P_n\to X_n$ with
$\|\hat{T_n}\|\le \|T_n\|+\epsilon\le \|T\|+\epsilon$ with $Q\circ
\hat{T_n}=T_n$. Piecing together this sequence of operators in the
obvious way will give us the desired lifting of $T$.
\end{proof}
\end{theorem}

Recall that if $\carda$ is finite or countably infinite then
$FBL(\carda)$ has a topological order unit as do finite dimensional
Banach lattices and $C(K)$-spaces. This gives us a source of
building blocks to create other projectives.

We already have some examples of Banach lattices which are not
projective.  It is interesting to note  that the free Banach
lattices on uncountably many generators seem to be, in some sense at
least, maximal projectives.

\begin{proexample}\label{freemaxproj}If $\carda$ is uncountable then there is no non-zero Banach lattice $X$ for which $X\oplus FBL(\carda)$ is projective under any norm.

\begin{proof}
Suppose that, under some norm, $FBL(\carda)\oplus X$ is projective,
where $X$ is a Banach lattice and $\carda$ is uncountable.

Consider $C(K)$, where $K=[0,\omega]\times[0,\omega_1]$, and (with the
notation of Example \ref{TPex}) $J=\{f\in C(K):f_{|A}\equiv 0\}$ so
that $C(K)/J$ is isometrically order isomorphic to $C(A)$.

For each $v\in V$ there is  $f_v\in C(A)$ with $0\le f_v(a)\le 1$
for all $a\in V$, $f_v(v)=1$ and $f_v$ identically zero on
$A\setminus V$. As $V$ has cardinality $\aleph_1$ there will be a
map of the set of generators $\{\delta_a:a\in \carda\}$ of
$FBL(\carda)$ onto $\{ f_v:v\in V\}$, which extends to a lattice
homomorphism of $FBL(\carda)$ into $C(A)$. The image of every
generator vanishes on $U$, hence the same is true for elements of
$T\big(FVL(\carda)\big)$ and, by continuity, for elements of
$T\big(FBL(\carda)\big)$. Note that $\bigcup_{f\in
FBL(\carda)}\{a\in A:f(a)\ne 0\}=V$.

As $U$ is an $F_\sigma$ there is $g\in C(A)$ with $g(u)>0$ for all
$u\in U$ and with $g$ identically zero on $A\setminus U$. If $X\oplus FBL(\carda)$ were projective and $x_0\in X_+\setminus\{0\}$ there would be a real-valued lattice homomorphism on $X\oplus FBL(\carda)$ with $\phi(x_0)>0$ (and necessarily $\phi_{|FBL(\carda)}\equiv 0$.) Define
$Sx=\phi(x)g$ for $x\in X$ so that $S$ is a lattice homomorphism of
$X$ into $C(A)$. The disjointness of the images of $S(X)$ and
$T\big(FBL(\carda)\big)$ shows that the direct sum operator $S\oplus
T:X\oplus FBL(\carda)\to C(A)=C(K)/J$ is also a lattice
homomorphism. If $X\oplus FBL(\carda)$ were projective we could find
a lattice homomorphism $\hat{S}\oplus \hat{T}:X\oplus FBL(\carda)\to
C(K)$ with $Q\circ (\hat{S}\oplus \hat{T})=S\oplus T$. The images of
$X\oplus\{0\}$ and $\{0\}\oplus FBL(\carda)$ will be disjoint in
$C(K)$ and their open supports will give disjoint open sets with
traces on $A$ equal to $U$ and $V$ respectively, which we know is
impossible.
\end{proof}
\end{proexample}

The family of projective Banach lattices seems to possess very few stability properties beyond those that we have already noted. In particular, 
closed sublattices of projectives need not be projective as the non-projective $c$ may be isometrically embedded as a closed sublattice of the 
projective Banach lattice $C([0,1])$, by mapping the sequence $(a_n)$ to the function that is linear on each interval $[1/(n+1),1/n]$ and takes the 
value $a_n$ at $1/n$. Similarly, we may realize $c$ as the quotient of $C([0,1])$ by the closed ideal $\{f\in C([0,1]):f(1/n)=0\forall n\in\Natural\}$, 
showing that the class of projective Banach lattices is not closed under quotients.

\section{Some Open Problems.}\label{problems}

We start with a few questions on free Banach lattices.

\begin{question}Must the norm on a free Banach lattice be
Fatou, or even Nakano? See \cite{Wi1} for the definition of a Nakano
norm.  We are not sure of the answer even when there are only
finitely many generators.
\end{question}

The following question is rather a long shot as we have very little evidence for it beyond the case of a finite number of generators (see below).

\begin{question}
If the free Banach lattice $FBL(\carda)$ is embedded as a closed ideal in a Banach lattice must it be a projection band?
\end{question}

The reason that this holds in the case of a finite number of generators is because this (isomorphic) property of Banach lattices is possessed by Banach lattices with a strong order unit. The following is undoubtedly well-known but we know of no convenient reference for it.

\begin{proposition}Let $Y$ be a Banach lattice with the property that every upward directed norm bounded subset of $Y_+$ is bounded above. If $Y$ is embedded as a closed ideal in a Banach lattice $X$ then it must be a projection band.
\begin{proof}
It suffices to prove that if $x\in X_+$ then the set $B=\{y\in Y:0\le y\le x\}$ has a supremum in $Y$. As $B$ is upward directed and norm bounded, it has an upper bound $u\in Y_+$. As $u\wedge x\in Y_+$, since $Y$ is an ideal, $u\wedge x$ is an upper bound for $B$ in $Y$. As we also have $0\le u\wedge x\le x$, $u\wedge x\in B$ so it is actually the maximum element of $B$.
\end{proof}
\end{proposition}

We have seen that, unless $|A|=1$, $FBL(A)^*$ is not an injective Banach lattice. However, in the case of \emph{finite} $A$, $FBL(A)^*$ is isomorphic to an AL-space and therefore to an injective Banach lattice.  We suspect that the following question might lead to another characterization of finitely generated free Banach lattices.

\begin{question}
When is $FBL(\carda)^*$ isomorphic to an injective Banach lattice?
\end{question}

In Theorem \ref{densitycharacter} we showed that the density character of $FBL(A)$ was equal to the cardinality of $A$ and related this to the density character or order intervals in $FBL(A)$. This is something of importance in the study of regular operators between Banach lattices, so an answer to the following question would have implications in that field.

\begin{question}Does every order interval in $FBL(A)$ have the same density character?
\end{question}

In the light of Theorem \ref{densitycharacter} that density character would have to be the cardinality of $A$.

\begin{question}Investigate the structure of the symmetric free norm
on $FBL(n)$.
\end{question}

\begin{question}Can the construction of a free Banach lattice be
generalized to give a free Banach lattice over a metric space? Here
a metric space $S$ embeds in a ``free'' Banach lattice in some sense
and any isometry of the generators into a Banach lattice extends to
a lattice homomorphism with some restriction on the norm. See
\cite{Pe}  for the Banach space case.
\end{question}

We have seen in Corollary \ref{ellinftyindual} that $FBL(\carda)^*$ contains a disjoint family of cardinality $\carda$ which contrasts strongly with the fact that disjoint families in $FBL(A)$ itself can only be at most countably infinite.

\begin{question}
How large can disjoint families of non-zero elements in $FBL(\carda)^*$ be?
\end{question}

 At present we
have no feel at all for what kinds of Banach lattice are likely to
be projective. Clearly, there are a lot of ``small'' ones, where
small means either separable or having a topological order unit. A
major and obvious question to pose is:

\begin{question}Determine the structure of the class of projective
Banach lattices.
\end{question}

In particular,

\begin{question}\label{atomproj}
Are atomic Banach lattices with an order continuous norm projective?
\end{question}

\begin{question}For what compact Hausdorff spaces $K$ is $C(K)$
projective under the supremum norm?
\end{question}

We know the answer to the preceding question for compact subsets of $\Real^n$ by Theorem \ref{fdCKs}.

The following two questions were posed by G. Buskes. An apparently simple question to answer is:

\begin{question}If $P_k$ ($1\le k\le n$) are projective Banach lattices with topological order units then is their $\ell_\infty$ sum also projective?
\end{question}

It is not difficult to lift a lattice homomorphism $T:\oplus_{k=1}^n P_k\to Y/J$ to a lattice homomorphism $\hat{T}:\oplus_{k=1}^n P_k\to Y$ by lifting the images of the topological order units first. The problem seems to be the norm condition on $\hat{T}$.

It is clear that the Fremlin tensor product, see \cite{Fr},  of two projective Banach lattices need not be projective in general. Example \ref{freemaxproj} shows that this cannot be true for the product of $\ell_1$ and $FBL(\carda)$ when $\carda$ is uncountable. There seems no good structural reason to expect a positive result to the next question, but a counterexample has eluded us so far.

\begin{question}If $X$ and $Y$ are projective Banach lattices with topological order units, is their Fremlin tensor product projective?
\end{question}

The building blocks that we can use in Theorem \ref{ellonesum} to
build new projectives include finite dimensional spaces,
$FBL(\carda)$ for $\carda$ either finite or countably infinite and
certain $C(K)$-spaces. Any of these, and the space that is produced
by that theorem, will be separable and hence will have a topological
order unit. Some (possibly rather rash) conjectures that we might
make are:

\begin{conjecture}If a projective Banach lattice has a topological
order unit then it is separable.
\end{conjecture}

\begin{conjecture}A projective Banach lattice which does not have a
topological order unit must be free.
\end{conjecture}

Even if this conjecture were to fail, we can ask for an improvement of Example \ref{freemaxproj} by asking:

\begin{conjecture}If $\carda$ is uncountable and a projective Banach lattice $X$ contains a closed ideal isomorphic to $FBL(\carda)$, do we actually have $X=FBL(\carda)$?
\end{conjecture}

\begin{question}
The $\ell_1$ sum of a sequence of finite dimensional Banach lattices is a Dedekind complete projective. Are these the only Dedekind $(\sigma$-)complete projectives?
\end{question}

\begin{conjecture}
All order continuous functionals on a projective Banach lattice
determined by its atoms.
\end{conjecture}

\begin{question}
Assuming a positive answer to Question \ref{atomproj}, we can
further ask if there is a result similar to Theorem \ref{ellonesum}
for $\ell_p$ sums ($1<p<\infty$) or for $c_0$ sums.
\end{question}

The whole of this paper has been written in an isometric setting.
All of our results may be reproved in an isomorphic setting, where
we replace an (almost) isometric condition on operators with mere
norm boundedness. It is not difficult to see that there will
automatically be uniform bounds to the norms of operators and that
isometrically free (resp. projective) Banach lattices will be
isomorphically free (resp. projective). Isomorphically free Banach
lattices will certainly be isomorphic to isometrically free Banach
lattices. At present it does not seem worth recording such a theory,
unless there is negative answer to the following question.

\begin{question}Is every isomorphically projective Banach lattice
isomorphic to an isometrically projective Banach lattice?
\end{question}

\affiliationone{B. de Pagter\\
Department of Mathematics,\\
Delft University of Technology,\\
Julianalaan 132, 2628 BL Delft,\\
 The Netherlands\\
\email{B.dePagter@tudelft.nl}}
\affiliationtwo{A.W. Wickstead\\
Pure Mathematics Research Centre,\\
Queens University Belfast,\\
Northern Ireland.\\
\email{A.Wickstead@qub.ac.uk}}

\begin{bibdiv}
\begin{biblist}[\resetbiblist{99}]

\bib{Ba}{article}{
   author={Baker, Kirby A.},
   title={Free vector lattices},
   journal={Canad. J. Math.},
   volume={20},
   date={1968},
   pages={58--66},
   issn={0008-414X},
   review={\MR{0224524 (37 \#123)}},
}
\bib{BM}{article}{
author={Banach, S.},
author={Mazur, S.},
title={Zuhr Theorie der linearen Dimension},
journal={Studia. Math.},
volume={4},
date={1933},
pages={100--112}}




\bib{Bl}{article}{
   author={Bleier, Roger D.},
   title={Free vector lattices},
   journal={Trans. Amer. Math. Soc.},
   volume={176},
   date={1973},
   pages={73--87},
   issn={0002-9947},
   review={\MR{0311541 (47 \#103)}},
}
\bib{Bo}{book}{
   author={Borsuk, Karol},
   title={Theory of retracts},
   series={Monografie Matematyczne, Tom 44},
   publisher={Pa\'nstwowe Wydawnictwo Naukowe, Warsaw},
   date={1967},
   pages={251},
   review={\MR{0216473 (35 \#7306)}},
}
\bib{Co}{article}{
   author={Conrad, Paul F.},
   title={Lifting disjoint sets in vector lattices},
   journal={Canad. J. Math.},
   volume={20},
   date={1968},
   pages={1362--1364},
   issn={0008-414X},
   review={\MR{0232720 (38 \#1043)}},
}


\bib{Di}{book}{
   author={Diestel, Joseph},
   title={Sequences and series in Banach spaces},
   series={Graduate Texts in Mathematics},
   volume={92},
   publisher={Springer-Verlag},
   place={New York},
   date={1984},
   pages={xii+261},
   isbn={0-387-90859-5},
   review={\MR{737004 (85i:46020)}},
}

\bib{DS}{book}{
   author={Dunford, Nelson},
   author={Schwartz, Jacob T.},
   title={Linear operators. Part I},
   series={Wiley Classics Library},
   note={General theory;
   With the assistance of William G. Bade and Robert G. Bartle;
   Reprint of the 1958 original;
   A Wiley-Interscience Publication},
   publisher={John Wiley \& Sons Inc.},
   place={New York},
   date={1988},
   pages={xiv+858},
   isbn={0-471-60848-3},
   review={\MR{1009162 (90g:47001a)}},
}
\bib{Fr}{article}{
   author={Fremlin, D. H.},
   title={Tensor products of Banach lattices},
   journal={Math. Ann.},
   volume={211},
   date={1974},
   pages={87--106},
   issn={0025-5831},
   review={\MR{0367620 (51 \#3862)}},
}
\bib{GJ}{book}{
   author={Gillman, Leonard},
   author={Jerison, Meyer},
   title={Rings of continuous functions},
   series={The University Series in Higher Mathematics},
   publisher={D. Van Nostrand Co., Inc., Princeton, N.J.-Toronto-London-New
   York},
   date={1960},
   pages={ix+300},
   review={\MR{0116199 (22 \#6994)}},
}
\bib{Ha}{article}{
   author={Haydon, Richard},
   title={Injective Banach lattices},
   journal={Math. Z.},
   volume={156},
   date={1977},
   number={1},
   pages={19--47},
   issn={0025-5874},
   review={\MR{0473776 (57 \#13438)}},
}


\bib{Lo}{article}{
   author={Lotz, Heinrich P.},
   title={Extensions and liftings of positive linear mappings on Banach
   lattices},
   journal={Trans. Amer. Math. Soc.},
   volume={211},
   date={1975},
   pages={85--100},
   issn={0002-9947},
   review={\MR{0383141 (52 \#4022)}},
}

\bib{LZ}{book}{
   author={Luxemburg, W. A. J.},
   author={Zaanen, A. C.},
   title={Riesz spaces. Vol. I},
   note={North-Holland Mathematical Library},
   publisher={North-Holland Publishing Co.},
   place={Amsterdam},
   date={1971},
   pages={xi+514},
   review={\MR{0511676 (58 \#23483)}},
}
\bib{MN}{book}{
   author={Meyer-Nieberg, Peter},
   title={Banach lattices},
   series={Universitext},
   publisher={Springer-Verlag},
   place={Berlin},
   date={1991},
   pages={xvi+395},
   isbn={3-540-54201-9},
   review={\MR{1128093 (93f:46025)}},
}

\bib{Mo}{article}{
   author={Moore, L. C., Jr.},
   title={The lifting property in Archimedean Riesz spaces},
   journal={Nederl. Akad. Wetensch. Proc. Ser. A 73=Indag. Math.},
   volume={32},
   date={1970},
   pages={141--150},
   review={\MR{0258707 (41 \#3353)}},
}
\bib{Or}{article}{
author={Orhon, Mehmet}, title={The ideal center of the dual of a
Banach lattice}, journal={Technische Universit\"at Darmstadt
preprint},volume={1182},date={1988}}

\bib{dPa}{article}{
   author={de Pagter, Ben},
   title={Irreducible compact operators},
   journal={Math. Z.},
   volume={192},
   date={1986},
   number={1},
   pages={149--153},
   issn={0025-5874},
   review={\MR{835399 (87d:47052)}},
   doi={10.1007/BF01162028},
}
\bib{Pe}{article}{
   author={Pestov, V. G.},
   title={Free Banach spaces and representations of topological groups},
   language={Russian},
   journal={Funktsional. Anal. i Prilozhen.},
   volume={20},
   date={1986},
   number={1},
   pages={81--82},
   issn={0374-1990},
   review={\MR{831059 (87h:22004)}},
}

\bib{RS}{article}{
   author={Ross, K. A.},
   author={Stone, A. H.},
   title={Products of separable spaces},
   journal={Amer. Math. Monthly},
   volume={71},
   date={1964},
   pages={398--403},
   issn={0002-9890},
   review={\MR{0164314 (29 \#1611)}},
}
\bib{Se}{book}{
   author={Semadeni, Zbigniew},
   title={Banach spaces of continuous functions. Vol. I},
   note={Monografie Matematyczne, Tom 55},
   publisher={PWN---Polish Scientific Publishers},
   place={Warsaw},
   date={1971},
   pages={584 pp. (errata insert)},
   review={\MR{0296671 (45 \#5730)}},
}
	
\bib{TL}{book}{
   author={Taylor, Angus E.},
   author={Lay, David C.},
   title={Introduction to functional analysis},
   edition={2},
   publisher={Robert E. Krieger Publishing Co. Inc.},
   place={Melbourne, FL},
   date={1986},
   pages={xii+467},
   isbn={0-89874-951-4},
   review={\MR{862116 (87k:46001)}},
}
\bib{Sp}{book}{
   author={Spanier, Edwin H.},
   title={Algebraic topology},
   note={Corrected reprint},
   publisher={Springer-Verlag},
   place={New York},
   date={1981},
   pages={xvi+528},
   isbn={0-387-90646-0},
   review={\MR{666554 (83i:55001)}},
}


\bib{To}{article}{
   author={Topping, David M.},
   title={Some homological pathology in vector lattices},
   journal={Canad. J. Math.},
   volume={17},
   date={1965},
   pages={411--428},
   issn={0008-414X},
   review={\MR{0174499 (30 \#4700)}},
}

\bib{We}{article}{
   author={Weinberg, Elliot Carl},
   title={Free lattice-ordered abelian groups},
   journal={Math. Ann.},
   volume={151},
   date={1963},
   pages={187--199},
   issn={0025-5831},
   review={\MR{0153759 (27 \#3720)}},
}
\bib{We2}{article}{
   author={Weinberg, Elliot Carl},
   title={Free lattice-ordered abelian groups. II},
   journal={Math. Ann.},
   volume={159},
   date={1965},
   pages={217--222},
   issn={0025-5831},
   review={\MR{0181668 (31 \#5895)}},
}


\bib{Wi}{article}{
   author={Wickstead, A. W.},
   title={Banach lattices with trivial centre},
   journal={Proc. Roy. Irish Acad. Sect. A},
   volume={88},
   date={1988},
   number={1},
   pages={71--83},
   issn={0035-8975},
   review={\MR{974286 (89m:46030)}},
}
\bib{Wi1}{article}{
   author={Wickstead, A. W.},
   title={An isomorphic version of Nakano's characterisation of $C_0(\Sigma)$},
   journal={Positivity},
   volume={11},
   date={2007},
   number={4},
   pages={609--615},
   issn={1385-1292},
   review={\MR{2346446 (2008g:46032)}},
   doi={10.1007/s11117-007-2078-6},
}



\bib{Wi2}{article}{
   author={Wickstead, A. W.},
   title={Banach lattices with topologically full centre},
   language={English, with English and Russian summaries},
   journal={Vladikavkaz. Mat. Zh.},
   volume={11},
   date={2009},
   number={2},
   pages={50--60},
   issn={1683-3414},
   review={\MR{2529410}},
}


\end{biblist}
\end{bibdiv}
\end{document}